\documentclass[12pt]{amsart}

\usepackage{amssymb}
\usepackage{bm}
\usepackage{graphicx}
\usepackage{psfrag}
\usepackage[usenames,dvipsnames]{xcolor}
\usepackage{enumerate}
\usepackage{multirow}
\usepackage{url}

\usepackage{subcaption}
\usepackage{comment}

\usepackage{algpseudocode}

\usepackage{mathtools}

\usepackage{xy}
\input xy
\xyoption{all}

\usepackage{tikz}

\newcommand{\excise}[1]{}%{$\star$\textsc{#1}$\star$}
\newtheorem{thm}{Theorem}[section]
\newtheorem{lemma}[thm]{Lemma}

\newtheorem{cor}[thm]{Corollary}
\newtheorem{prop}[thm]{Proposition}
\newtheorem{conj}[thm]{Conjecture}
\newtheorem{question}[thm]{Question}

\theoremstyle{definition}

\newtheorem{example}[thm]{Example}
\newtheorem{remark}[thm]{Remark}
\newtheorem{defn}[thm]{Definition}

\numberwithin{equation}{section}

\renewcommand\>{\rangle}
\newcommand\<{\langle}

\newcommand\RR{\mathbb{R}}

\newcommand\ZZ{\mathbb{Z}}

\newcommand\kk{\Bbbk}

\newcommand\cC{{\mathcal C}}

\newcommand\app{\mathord\approx}

\newcommand\til{\mathord\sim}

\DeclareMathOperator\Span{span} % span

\DeclareMathOperator\Ap{Ap} % Apery set
\DeclareMathOperator\supp{supp} % support
\begin{document}%%%%%%%%%%%%%%%%%%%%%%%%%%%%%%%%%%%%%%%%%%%%%%%%%%%%%%%%
%%%%%%%%%%%%%%%%%%%%%%%%%%%%%%%%%%%%%%%%%%%%%%%%%%%%%%%%%%%%%%%%%%%%%%%%

\mbox{}
%\vspace{-2ex}%-1.1743pt}
\title[Numerical semigroups, polyhedra, and posets IV]{Numerical semigroups, polyhedra, and posets IV:\ \\ walking the faces of the Kunz cone}

\author[Brower]{Cole Brower}
\address{Mathematics Department\\San Diego State University\\San Diego, CA 92182}
\email{cole.brower@gmail.com}

\author[McDonough]{Joseph McDonough}
\address{School of Mathematics\\University of Minnesota\\Minneapolis, MN 55455}
\email{mcdo1248@umn.edu}

\author[O'Neill]{Christopher O'Neill}
\address{Mathematics Department\\San Diego State University\\San Diego, CA 92182}
\email{cdoneill@sdsu.edu}

\date{\today}

\begin{abstract}
A numerical semigroup is a cofinite subset of $\mathbb Z_{\ge 0}$ containing $0$ and closed under addition.  Each numerical semigroup $S$ with smallest positive element $m$ corresponds to an integer point in the Kunz cone $\mathcal C_m \subseteq \mathbb R^{m-1}$, and the face of $\mathcal C_m$ containing that integer point determines certain algebraic properties of $S$.  In this paper, we introduce the Kunz fan, a pure, polyhedral cone complex comprised of a faithful projection of certain faces of $\mathcal C_m$.  We characterize several aspects of the Kunz fan in terms of the combinatorics of Kunz nilsemigroups, which are known to index the faces of $\mathcal C_m$, and our results culminate in a method of ``walking'' the face lattice of the Kunz cone in a manner analogous to that of a Gr\"obner walk.  We apply our results in several contexts, including a wealth of computational data obtained from the aforementioned ``walks'' and a proof of a recent conjecture concerning which numerical semigroups achieve the highest minimal presentation cardinality when one fixes the smallest positive element and the number of generators.  
\end{abstract}

\maketitle

% \setcounter{tocdepth}{1}
% \tableofcontents

%%%%%%%%%%%%%%%%%%%%%%%%%%%%%%%%%%%%%%%%%%%%%%%%%%%%%%%%%%%%%%%%%%%%%%%%%
\section{Introduction}%%%%%%%%%%%%%%%%%%%%%%%%%%%%%%%%%%%%%%%%%%%%%%%%%%%
\label{sec:intro}%%%%%%%%%%%%%%%%%%%%%%%%%%%%%%%%%%%%%%%%%%%%%%%%%%%%%%%%
%raggedbottom%%%%%%%%%%%%%%%%%%%%%%%%%%%%%%%%%%%%%%%%%%%%%%%%%%%%%%%%%%%%

% Numerical semigroups
A numerical semigroup is a cofinite subset $S \subseteq \ZZ_{\ge 0}$ containing $0$ and closed under addition; see~\cite{numericalappl} for a thorough intro.  
% Numerical semigroups have been studied since the 1960's, and have applications in combinatorics, algebraic geometry, and number theory.  For a thorough introduction to the theory of numerical semigroups, see~\cite{numerical}.  
% Kunz cone and history
Each numerical semigroup with smallest positive element $m$ corresponds to an integer point in the Kunz cone $\cC_m \subseteq \RR^{m-1}$.  Inspired by a construction of Kunz~\cite{kunz}, this family of polyhedral cones has been of significant interest in the last decade.  The original motivation for investigating these polyhedra was enumerative in nature (e.g., utilizing lattice point methods to enumerate numerical semigroups with a given number of gaps~\cite{kunzcoords}, or addressing some longstanding asymptotic questions~\cite{kaplancounting}).  

Much of the recent interest in the Kunz cone, however, has focused on the faces of~$\cC_m$.  In this time, numerical semigroups $S$ and $T$ corresponding to points in the same face of $\cC_m$ have been shown to share numerous algebraic properties, including embedding dimension (i.e. the number of minimal generators), Cohen-Macaulay type, and the symmetric property \cite{wilfmultiplicity,kunzfaces1}.
% .  Embedding dimension (i.e., the number of minimal geneators) and Cohen-Macauley type, as well as the symmetric property, are relatively easily shown to be shared by $S$ and $T$~\cite{wilfmultiplicity,kunzfaces1}.  
The defining toric ideals of $S$ and $T$ have been shown to possess similar minimal binomial generators~\cite{kunzfaces3}, and were recently shown to have identical Betti numbers up to reduction of graded degrees modulo $m$~\cite{kunzfiniteres}.  When $S$ and $T$ lie in certain popular families of numerical semigroups, such as those that are complete intersection, generated by (generalized) arithmetic sequences, or constructed via gluings, has also been shown to coincide~\cite{kunzfaces2}.  

At the heart of these shared properties is the Kunz nilsemigroup:\ a finite nilsemigroup associated to each face $F \subseteq \cC_m$, derived from a portion of the divisibility poset of each numerical semigroup corresponding to a point in $F$ \cite{kunzfaces3,kunzfaces1}.  Kunz nilsemigroups provide a combinatorial framework for working with both the geometry of the faces of the Kunz cone~\cite{adventure5} and the algebra of the numerical semigroups therein~\cite{minprescard}.  
% Their discovery has also been used to make headway on questions surrounding minimal presentation cardinality~\cite{minprescard}.  

One application of the Kunz cone is a computational method of enumerating numerical semigroups with fixed $m$; this was utilized in~\cite{wilfmultiplicity} to make headway on Wilf's conjecture, one of the longest open problems in the numerical semigroups literature.  Unfortunately, the number of faces of $\cC_m$ grows quickly in $m$, so computing the full face lattice of $\cC_m$ quickly becomes prohibitively difficult.  If one is only interested in numerical semigroups with a relatively small number of generators (as is often the case), the relevant faces of $\cC_m$ have small dimension and thus are far less numerous.  However, since the cone $\cC_m$ is defined by hyperplanes, existing computational methods make enumerating only lower-dimensional faces difficult.  

The aim of the present manuscript is twofold:\ (i)~to introduce a method of computing and enumerating, for fixed $k$, the faces of~$\cC_m$ in which the points corresponding to $k$-generated numerical semigroups reside; and (ii)~to further develop the theoretical framework linking Kunz nilsemigroups and the geometry of the faces of $\cC_m$ they index.  
To this end, we introduce the Kunz fan (Definition~\ref{d:kunzfan}), a pure cone complex comprised of a faithful projection of such faces.  We~characterize several aspects of the Kunz fan, such as its boundary (Theorem~\ref{t:purefan} and Proposition~\ref{p:fanboundary}), its chambers (Corollary~\ref{c:chamberstaircase}), an $H$-description of its faces (Theorem~\ref{t:outerbettihdesc}), and chamber incidence (Theorem~\ref{t:cutandpaste}), in terms of the combinatorics of Kunz nilsemigroups.  

On the computational front, our results yield a method of ``walking'' the face lattice of the Kunz cone (Section~\ref{sec:kunzwalk}) in a manner analogous to that of a Gr\"obner walk~\cite{toricgroebnerwalk} (this comparison is not just superficial; see Remark~\ref{r:groebnerfans}).  Our algorithm represents a marked improvement over prior methods of enumerating the faces of $\cC_m$:\ we computed all faces of every Kunz fan with $k = 3$ and $m \le 42$ on our personal machines in under a day (we ran out of memory at $m = 43$), whereas computing the full face lattice for $m = 18$ for~\cite{wilfmultiplicity} took a cluster 3 weeks.  This data is of particularly high interest in examining the relationship between a numerical semigroup's multiplicity, embedding dimension, and minimal presentation cardinality; indeed, these methods were utilized to obtain much of the table given in the introduction of~\cite{minprescard}, and several of the constructions given thereafter were identified from the extremal examples produced by those computations.  

On the theoretical front, Section~\ref{sec:applications} includes several applications:\ 
we classify the Kunz nilsemigroups of 3-generated numerical semigroups; 
we identify a family of faces of $\cC_m$ that yields an exponential lower bound on the number of rays of $\cC_m$; 
we prove \cite[Conjecture~7.3]{minprescard} concerning which numerical semigroups with fixed $m$ and $k$ achieve the highest minimal presentation cardinality; 
% we give a discrete-geometric characterization of which finite nilsemigroups correspond to faces of $\cC_m$; 
and we prove a result related to a longstanding open problem concerning Gr\"obner fans of toric ideals~\cite{groenberfanchamberfacets}.  We also characterize the finite nilsemigroups that are Kunz (Theorem~\ref{t:kunzcondition}), answering a question posed in~\cite{kunzfaces1}.

%%%%%%%%%%%%%%%%%%%%%%%%%%%%%%%%%%%%%%%%%%%%%%%%%%%%%%%%%%%%%%%%%%%%%%%%%
\section{Background}%%%%%%%%%%%%%%%%%%%%%%%%%%%%%%%%%%%%%%%%%%%%%%%%%%%%%
\label{sec:background}%%%%%%%%%%%%%%%%%%%%%%%%%%%%%%%%%%%%%%%%%%%%%%%%%%%
%raggedbottom%%%%%%%%%%%%%%%%%%%%%%%%%%%%%%%%%%%%%%%%%%%%%%%%%%%%%%%%%%%%

In this section, we recall some necessary background on semigroups, polyhedral geometry, Kunz nilsemigroups, and the Kunz cone. For a more complete overview, see~\cite{numericalappl}, \cite{ziegler}, \cite[Section~2]{minprescard} and \cite[Section~3]{kunzfaces3}, respectively.  

For a commutative semigroup $(N, +)$, an element $\infty\in N$ is \emph{nil} if $a + \infty = \infty$ for all $a \in N$. We call an element $a \in N$ \emph{nilpotent} if $na = \infty$ for some $n \in \ZZ_{\geq 1}$, and \emph{partly cancellative} if $a + b = a + c \ne \infty$ implies $b = c$ for all $b,c \in N$. We say $N$ is a \emph{nilsemigroup} if $N$ has an identity element and every non-identity element is nilpotent, and that $N$ is \emph{partly cancellative} if every non-nil element of $N$ is partly cancellative.  Note that any nilsemigroup is \emph{reduced}, meaning its only unit is the identity.  We call any element of $N$ that cannot be written as the sum of two non-identity elements an \emph{atom} of $N$. 

All semigroups $N$ in this paper are assumed to be commutative, partly cancellative, finitely generated, and reduced.  Under these assumptions, the atoms $n_0, \ldots, n_k \in N$ comprise the unique minimal generating set~\cite{fingenmon}; we denote this by $N = \< n_0, \ldots, n_k\>$.  A \emph{factorization} of $n \in N$ is an expression
\[
n = z_0n_0 + \cdots + z_kn_k  
\]
where each $z_i \in \ZZ_{\geq 0}$. The \emph{set of factorizations} of $n \in N$ is the set
\[
\mathsf{Z}_N(n) = \{z \in \ZZ_{\geq 0}^{k+1} \mid n = z_0n_0 + \cdots + z_kn_k\} \subset \ZZ_{\geq 0}^{k+1}.
\]
The \emph{factorization homomorphism} $\varphi_N : \ZZ_{\geq 0}^{k+1} \to N$ is the semigroup homomorphism 
\[
\varphi_N(z_0,\ldots, z_k) = z_0n_0 + \cdots z_kn_k.
\]
The \emph{kernel} of $\varphi_N$ is 
\[
\ker \varphi_N = \left\{ (a,b) \in \ZZ_{\geq 0}^{k+1} \times \ZZ_{\geq 0}^{k+1} \mid \varphi_N(a) = \varphi_N(b) \right\}  
\]
which induces a \emph{congruence relation} $\til$ on $\ZZ_{\geq 0}^{k+1}$, setting $a \sim b$ whenever $(a,b) \in \ker \varphi_N$ (recall that a \emph{congruence} is an equivalence relation such that $a \sim b$ implies $a+c \sim b+c$ for any $a,b,c \in \ZZ_{\geq 0}^{k+1}$).  We call any such pair $(a,b)$ or $a \sim b$ a \emph{trade} of $N$.  
A set of trades $\rho$ is said to \emph{generate} $\til$ if $\til$ is the smallest congruence on $\ZZ_{\geq 0}^{k+1}$ containing $\rho$.  A~\emph{presentation} of $N$ is a set $\rho$ of trades obtained by taking a generating set for $\til$ and omiting any $a \sim b$ where $\varphi_N(a)$ is nil.  A presentation $\rho$ of $\til$ is \emph{minimal} if no proper subset of $\rho$ is a presentation of $\til$.  It is known that any two minimal presentations of a finitely generated, partly cancellative semigroup have the same cardinality~\cite{kunzfaces3,fingenmon}.  

The \emph{support} of a factorization $ z \in \ZZ_{\geq 0}^{k+1}$ is the set 
$$\supp(z) = \{i \mid z_i > 0\}$$
of nonzero coordinates.  For $Z \subseteq \ZZ_{\geq 0}^{k+1}$, define 
$$\supp(Z) = \{i \mid z_i > 0 \text{ for some } z \in Z\},$$
and the \emph{factorization graph} $\nabla_Z$, whose vertices are elements of $Z$, and two factorizations $z, z' \in Z$ are connected by an edge if $\supp(z) \cap \supp(z') \ne \emptyset$. For $n \in N$, we write $\nabla_n$ for the factorization graph whose vertex set is $\mathsf Z_N(n)$.  For each $i \in \supp(Z)$, let
\[
Z - e_i = \{z-e_i \mid z \in Z, \, i \in \supp(z)\}.
\]

Suppose $N$ is a finite, partially cancellative nilsemigroup.  An \emph{outer Betti element} of~$N$ is a subset $B \subseteq \mathsf{Z}_N(\infty)$ such that
\begin{enumerate}[(i)]
\item for each $i \in \supp(B)$, $B - e_i = \mathsf{Z}_N(p)$ for some $p \in N \setminus \{\infty\}$, and
\item the graph $\nabla_B$ is connected.
\end{enumerate}

A \emph{numerical semigroup} $S$ is an additive subsemigroup of $(\ZZ_{\geq 0}, +)$ that is cofinite and contains $0$.  Numerical semigroups have a unique minimal generating set, the size of which we call the \emph{embedding dimension}, and the smallest element of which we call the \emph{multiplicity}. Letting $m$ be an element of $S$, the Ap\'ery set of $S$ is the set 
\[
\Ap(S;m) = \{n \in S \mid n - m \notin S\}
\]
containing the smallest element of $S$ in each equivalence class modulo $m$.  
% Writing $\Ap(S);m = \{w_0, \ldots, w_{m-1}\}$ with $w_i \equiv i \bmod m$.  
Let $\approx$ denote the congruence on $S$ setting $a \approx b$ whenever $a = b$ or $a,b \notin \Ap(S;m)$.  The~quotient semigroup $S/\app$ is a nilsemigroup with one non-nil element for each element of $\Ap(S;m)$.  The \emph{Kunz nilsemigroup} of $S$ is given by $N = \ZZ_m \cup \{\infty\}$ as sets, and is obtained from $S/\app$ by replacing each non-nil element by its equivalence class in $\ZZ_m$.  

\begin{thm}[{\cite{kunzfaces3}}]\label{t:minprescard}
If $\rho$ is a minimal presentation of a numerical semigroup $S$, then $|\rho| = |\rho'| + \beta$, where $\rho'$ is any minimal presentation for the Kunz nilsemigroup $N$ of $S$, and $\beta$ is the number of outer Betti elements of $N$.  
\end{thm}

We briefly illustrate Theorem~\ref{t:minprescard} and the definition preceding it in the following example.  See~\cite{minprescard,kunzfaces3} for a more thorough introduction to outer Betti elements and their relationship to minimal presentations of numerical semigroups.  

\begin{example}\label{e:outerbettis}
Let $S = \<13, 53, 15, 35\>$.  One can verify computationally~\cite{numericalsgpsgap} that
\[
\Ap(S;13) = \{0, 53, 15, 68, 30, 70, 45, 85, 60, 35, 75, 50, 90\},
\]
where the elements are listed in order of their equivalence classes modulo $m = 13$.  
The~partially ordered set~(c) depicted in Figure~\ref{f:3dfan} encodes the divisibility relations of the non-nil elements of Kunz nilsemigroup $N$ of $S$.  For instance, 3 precedes~1 in the diagram since $68 - 53 = 15 \in S$, but 3 does not precede 5 since $70 - 68 = 2 \notin S$.  Moreover, $N = \<1,2,9\>$, as these are the elements covering 0.  

Write $\Ap(S;13) = \{0, a_1, \ldots, a_{12}\}$ with each $a_i \equiv i \bmod 13$.  One can check that $\mathsf Z_N(11) = \{(0,1,1)\}$ since $a_{11} = 50 = 15 + 35 = a_2 + a_9$, and in fact this is the only factorization of $a_{11}$.  Additionally, $(1,2,0) \in \mathsf Z_N(\infty)$ since $a_1 + 2a_2 = 83 > 70 = a_4$.  

There are 6 outer Betti elements, each consisting of a single factorization from among
\[
(2,0,0), 
\quad
(1,2,0), 
\quad
(0,7,0),
\quad
(1,0,1), 
\quad
(0,2,1), 
\quad \text{and} \quad
(0,0,3),
\]
and each can be seen as a factorization in $\mathsf Z_N(\infty)$ that is minimal under the component-wise partial order.  For instance, 
\[
\mathsf Z_N(11) = \{(0,1,1)\}, 
\qquad
\mathsf Z_N(4) = \{(0,2,0)\},
\qquad \text{and} \qquad
\{(0,2,1)\} \subseteq \mathsf Z_N(\infty),
\]
imply $\{(0,2,1)\}$ is an outer Betti element of $N$.  
\end{example}

\begin{figure}[t!]
\begin{center}
\includegraphics[height=1.8in]{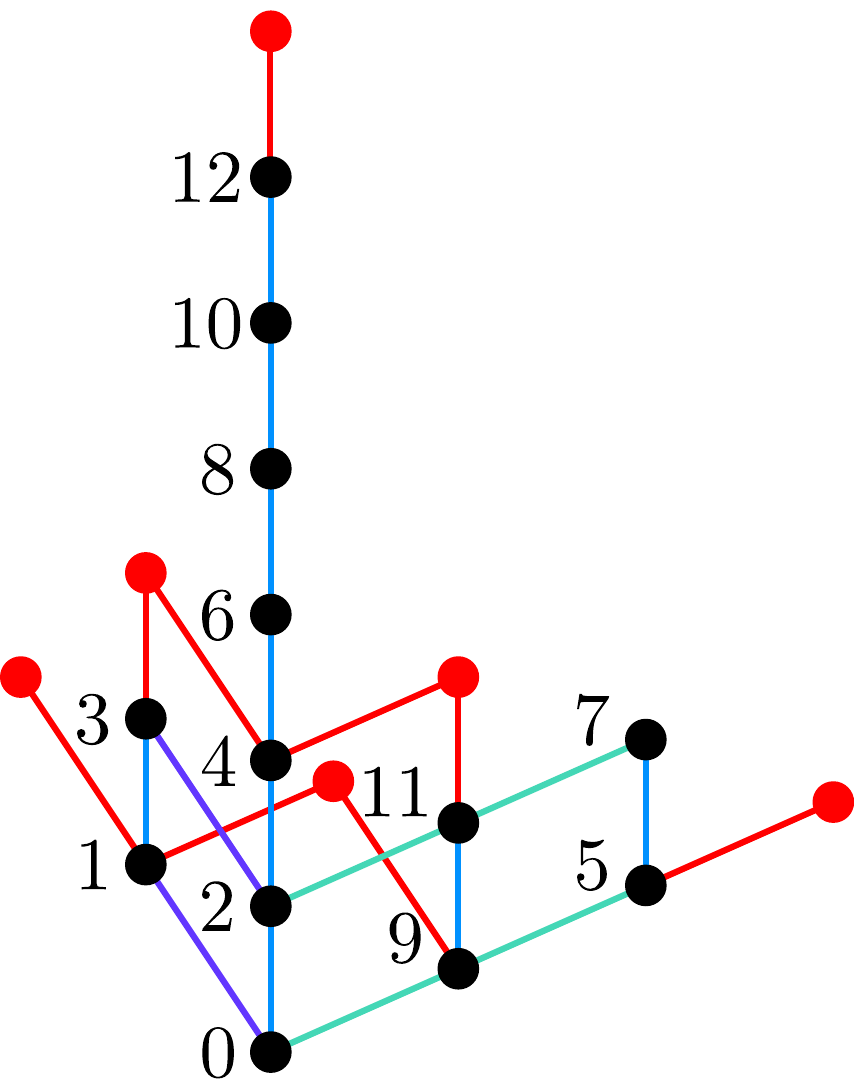}
% \hspace{3em}
% \includegraphics[height=0.9in]{ex-nonkunznilsemigroup.pdf}
\end{center}
\caption{Kunz poset in Example~\ref{e:outerbettis} with outer Betti elements in red.} 
\label{f:outerbettis}
\end{figure}

A \emph{rational polyhedral cone} $C \subseteq \RR^d$ is the set of solutions to a finite set of linear inequalties with rational coefficients, i.e.
\[
C = \{x \in \RR^d \mid Ax \ge 0\}
\]
for some rational matrix $A$.  We say $C$ is \emph{strongly convex} (or \emph{pointed}) if it does not contain any positive dimension linear subspace of $\RR^d$, and the \emph{dimension} of $C$ is the vector space dimension $\dim C = \dim_\RR \Span_{\RR}C$. If none of the rows of $A$ can be omitted without changing $C$, we call $A$ the \emph{H-description} or \emph{facet description} of $C$, and refer to each inequality therein as a \emph{defining inequality} or \emph{facet inequality} of $P$.  If $\dim C = d$, then the facet description is unique up to the reordering and scaling of rows.  
Given a facet inequality $a_1x_1 + \cdots + a_dx_d \geq 0$ of $C$, the intersection of $C$ with the hyperplane $a_1x_1 + \cdots + a_dx_d = 0$ is called a \emph{facet} of $C$. Any intersection of facets is called a \emph{face} of~$C$; note that the condition of strong convexity is equivalent to $\{0\}$ being a face of~$C$. Given a face $F \subseteq C$, the \emph{relative interior} of $F$, denoted $F^{\circ}$, is the set of points in $F$ that do not lie in any proper faces of $F$.
A finite collection $\mathcal F$ of polyhedral cones is called a \emph{polyhedral fan} if 
\begin{enumerate}[(i)]
\item 
for any $C \in \mathcal{F}$, every face of $C$ is also in $\mathcal{F}$, and

\item 
the intersection of any $C,D \in \mathcal{F}$ is a face of both and lies in $\mathcal F$.

\end{enumerate}
The elements of a fan $\mathcal F$ are often called its \emph{faces}.  
A fan is \emph{pure} if its maximal faces (under containment) have the same dimension, and in this case, we refer to the maximal faces as \emph{chambers}.  

For $m \geq 2$, the strongly convex cone $\cC_m \subseteq \RR^{m-1}_{\geq 0}$ with facet inequalities
\[
x_i + x_j \geq x_{i+j}
\qquad \text{for} \qquad
i,j \in \ZZ_m \setminus \{0\}
\qquad \text{with} \qquad
i + j \ne 0
\]
is called the \emph{Kunz cone}.  
% where coordinates are indexed by $\ZZ_m \setminus \{0\}$.  
A point $z = (z_1, \ldots, z_{m-1}) \in \cC_{m} \cap \ZZ^{m-1}$ is an \emph{Ap\'ery point} if each $z_i \equiv i \bmod m$.  
The following is the culmination of results from~\cite{kunzfaces3,kunzfaces1,kunz}.  

\begin{thm}\label{t:kunznilsemigroupface}
The Ap\'ery points of $\cC_m$ are in bijective correspondence with numerical semigroups containing $m$; more specifically, $z \mapsto S$ where $\Ap(S;m) = \{0, z_1,\ldots, z_{m-1}\}$.  
% Two numerical semigroups correspond to Ap\'ery points in the same face of $\cC_m$ if and only if they have identical Kunz nilsemigroups.  
Fix a face $F \subseteq \cC_m$, the set
$$H = \{0\} \cup \{i : x_i = 0 \text{ for all } x \in F\} \subseteq \ZZ_m$$
is a subgroup of $\ZZ_m$ (called the \emph{Kunz subgroup} of $F$).  Define a binary operation $\oplus$ on~$N = (\ZZ_m/H) \cup \{\infty\}$ so that $\infty$ is nil, and for any nonzero $a,b \in \ZZ_m$,
\[
\overline a \oplus \overline b = \begin{cases}
a+b & \text{if } x_a + x_b = x_{a+b} \text{ for all } x \in F; \\
\infty & \text{otherwise.}
\end{cases}
\]
Under the above definition, $(N, \oplus)$ is a partly cancellative nilsemigroup (called the \emph{Kunz nilsemigroup} of $F$), and if the Ap\'ery point of a numerical semigroup $S$ lies in $F$, then $H = \{0\}$ and the Kunz nilsemigroup of $S$ equals the Kunz nilsemigroup of $F$.  
\end{thm}

In view of the above theorem, we say a numerical semigroup $S$ lies in the face $F \subseteq \cC_m$, and write $S \in F$, if the Ap\'ery point corresponding to $S$ lies in $F$.

%%%%%%%%%%%%%%%%%%%%%%%%%%%%%%%%%%%%%%%%%%%%%%%%%%%%%%%%%%%%%%%%%%%%%%%%%
\section{Kunz fans}%%%%%%%%%%%%%%%%%%%%%%%%%%%%%%%%%%%%%%%%%%%%%%%%%%%%%%
\label{sec:kunzfan}%%%%%%%%%%%%%%%%%%%%%%%%%%%%%%%%%%%%%%%%%%%%%%%%%%%%%%
%raggedbottom%%%%%%%%%%%%%%%%%%%%%%%%%%%%%%%%%%%%%%%%%%%%%%%%%%%%%%%%%%%%

Throughout this section, fix $A = \{p_1, \ldots, p_k\} \subseteq \ZZ_m \setminus \{0\}$ with $\gcd(A, m) = 1$, and let $p:\RR^{m-1} \to \RR^k$ denote the linear map that projects each point in $\RR^{m-1}$ onto the coordinates whose indices lie in $A$, i.e., 
$$p(y) = (y_{p_1}, \ldots, y_{p_k}).$$

\begin{defn}\label{d:kunzfan}
Let $\mathcal F(m;A)$ denote the set of faces $F \subseteq \cC_m$ for which each atom of the Kunz nilsemigroup $N$ of $F$ has a representative in $A$.  
The \emph{Kunz fan} of $A$ is the set
$$
\mathcal G(m;A) = \{p(F) : F \in \mathcal F(m;A)\}
$$
of cones in $\RR^k$ (we prove in Theorem~\ref{t:purefan} that $\mathcal G(m;A)$ is indeed a fan).  Note that under this definition, $|A|$ may exceed the number of atoms of $N$, but this allowance is necessary to ensure $\mathcal G(m;A)$ is a fan; see Example~\ref{e:2dfan} for a discussion of this subtlety.  
\end{defn}

One of the main results of this section is that $\mathcal G(m;A)$ is pure, and that the maximal faces of $\mathcal G(m;A)$ with respect to containment are precisely those whose corresponding nilsemigroup lies in the following family.  

\begin{defn}\label{d:staircaseposet}
A finite, partly cancellative nilsemigroup $N$ is a \emph{staircase} nilsemigroup if every non-nil element uniquely factors as a sum of atoms, and a numerical semigroup $S$ is called \emph{staircase} if its Kunz nilsemigroup is staircase.  As each outer Betti element of a staircase nilsemigroup $N$ consists of a single factorization of $\infty$, when there can be no confusion we refer to each such factorization as an outer Betti element of $N$.  
\end{defn}

A numerical semigroup $S$ is staircase if and only if every element of $\Ap(S;m)$ factors uniquely.  
In this case, $S$ is said to have \emph{Ap\'ery set of unique expression}; such semigroups have been studied in the literature~\cite{rosalesapery} and were central to the constructions in~\cite{minprescard}.  

The examples in this section examine two Kunz fans in detail, and reference results in this section and subsequent sections.  We encourage the reader to peek ahead at these results while reading these examples, as they put the landscape of Sections~\ref{sec:kunzfan}, \ref{sec:facetouterbetti}, and~\ref{sec:kunzwalk} in perspective.  

\begin{example}\label{e:2dfan}
Let $m = 20$ and $A = \{6,11\}$.  The Kunz fan $\mathcal G(m;A)$ is depicted in Figure~\ref{f:2dfan} alongside the staircase Kunz nilsemigroups corresponding to its chambers, with outer Betti elements depicted in red.  Let $N$ denote the Kunz nilsemigroup corresponding to chamber~(b).  As~we will see in Proposition~\ref{p:circleoflightsinverse} below, any point $(x_6,x_{11})$ in the interior of chamber~(b) must satisfy $7x_6 > 2x_{11}$, as $7 \cdot 6 \equiv 2 \bmod 20$ and $2 \cdot 11 \equiv 2 \bmod 20$, but $(0,2) \in \mathsf Z_N(2)$ while $\{(7,0)\}$ is an outer Betti element of $N$.  

Let's examine the facets on the boundary of the fan $\mathcal G(m;A)$, as these each identify subtleties in Definition~\ref{d:kunzfan}.  One is defined by $x_6 \le 6x_{11}$, which must be satisfied by every point in a face of $\mathcal G(m;A)$ by Proposition~\ref{p:circleoflightsinverse} since for any point $y \in \cC_m$ that projects to a point $(x_6, x_{11})$ in a face of $\mathcal G(m;A)$, each coordinate $y_i$ is a non-negative integral combination of $x_6$ and $x_{11}$.  Note that the Kunz nilsemigroup $N$ of this facet has $N = \<11\>$, since $x_6 = 6x_{11}$ implies 6 is a multiple of 11 in $N$; this is why we do not require $N$ to possess an atom for each element of $A$ in Definition~\ref{d:kunzfan}.  

The other facet is defined by $x_6 \ge 0$; since $\gcd(6,20) > 1$ and $\gcd(6,11,20) = 1$, for each $x_{11} > 0$ one can locate points $(x_6, x_{11})$ in the interior of chamber~(a) wherein $x_{6}$ is arbitrarily small.  The Kunz subgroup $H$ of this facet is nontrivial since $x_6 = 0$ for every point therein; this is why we only require each atom of the Kunz poset $N$ to have a representative in $A$ in Definition~\ref{d:kunzfan}, rather than requiring $A$ to equal the set of atoms of $N$.  
\end{example}

\begin{figure}[t!]
\begin{center}
\includegraphics[width=6in]{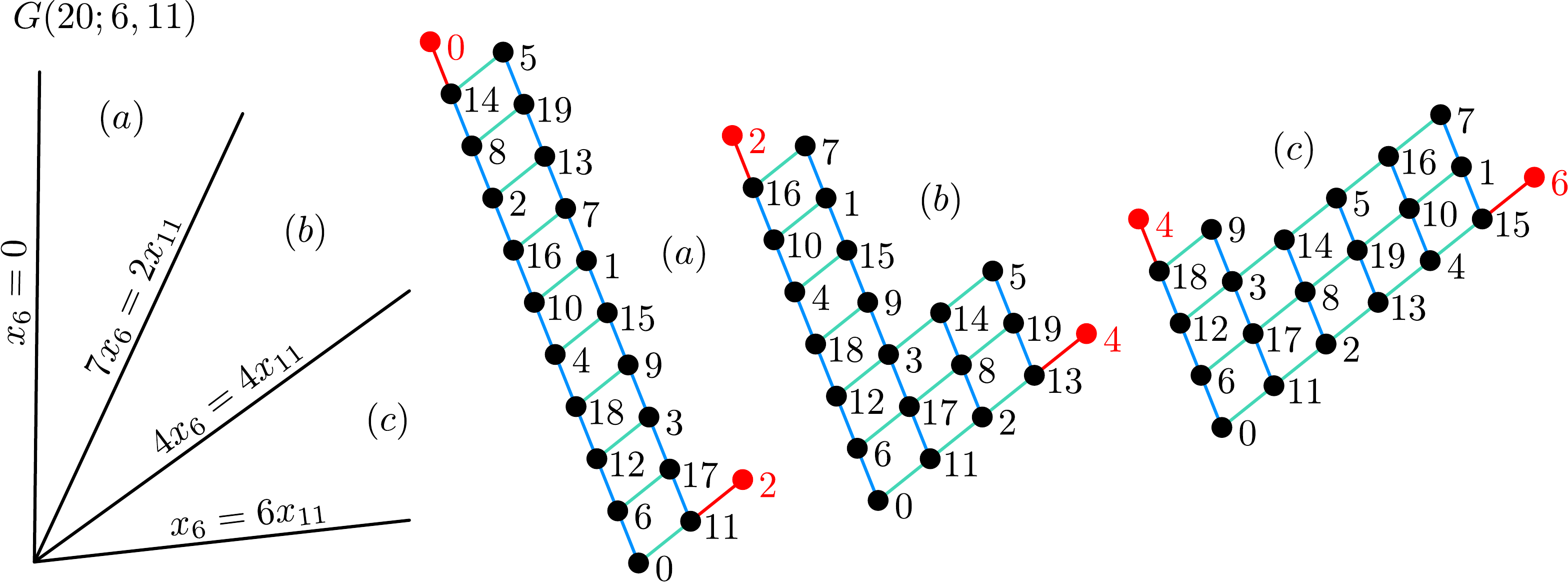}
\end{center}
\caption{Posets from the chambers of the Kunz fan $\mathcal G(20; 6, 11)$.}
\label{f:2dfan}
\end{figure}

Consider the piecewise linear map $q:\RR_{\ge 0}^k \to \RR_{\ge 0}^{m-1}$ given by $q(x) = y$, where
$$y_i = \min\{c \cdot x \mid c \in \ZZ_{\ge 0}^k \text{ with } c_1p_1 + \cdots + c_kp_k \equiv i \bmod m\}.$$
% The following result implies $q$ is piecewise linear on $\RR_{\ge 0}^k$ and describes the behavior of $q$ on the cones in $\mathcal G(m;A)$.  

\begin{prop}\label{p:circleoflightsinverse}
Fix a face $F \in \mathcal F(m; A)$ with Kunz nilsemigroup $N$ and Kunz subgroup $H$.  The map $p$ is injective on $F$, and if $x \in p(F)$, then $y = q(x)$ satisfies
$$y_i = z_1 x_1 + \cdots + z_k x_k$$
for any factorization $z \in \mathsf Z_N(i)$.  In particular, $q$ restricts to a linear map on $p(F)$, and $q(p(y)) = y$ for every $y \in F$.  
\end{prop}

\begin{proof}
For each nonzero $i \in \ZZ_m$, Theorem~\ref{t:kunznilsemigroupface} implies one of the following must hold:\ (i) $i \in H$ and $y_i = 0$; (ii) $i \in a + H$ for some $a \in A$ and $y_i = y_a$; or (iii) for any factorization $z \in \mathsf Z_N(i)$, we have 
$$y_i = z_1 y_{p_1} + \cdots + z_k y_{p_k}$$
for all $y \in F$.  As such, for any $y, y' \in F$, if $p(y) = p(y')$, then $y_a = y_a'$ for all $a \in A$, and thus $y = y'$.  This proves $p$ is injective on $F$.  

Next, suppose $p(y) = x$, and let $y' = q(x)$.  Fix $i$, and suppose that $c \in \ZZ_{\ge 0}^k$ satisfies $y_i' = c \cdot x$.  We claim $c \in \mathsf Z_N(i)$.  Indeed, if $i = 0$, then $y_i = 0$, and if $i \in a + H$ for some $a \in A$, then $y_i = y_a$.  For all other cases, fix $j$ with $c_j > 0$.  This means $i + H$ covers $(i - p_j) + H$ in $N$, and by minimality of $c \cdot x$, $(c - e_j) \cdot x = y_{i - p_j}'$.  By induction, we may assume $c - e_j \in \mathsf Z_N(j)$, and so we have $c = (c - e_j) + e_j \in \mathsf Z_N(i)$ since $y_{i - p_j}' + y_{p_j} = y_i'$.  This proves the claim.  By the preceding paragraph, we now have $y' = y$, and the remaining claims all immediately follow.  
\end{proof}

It is not hard to show that $q(\RR_{\ge 0}^k) \subseteq \cC_m$, although the injectivity in Proposition~\ref{p:circleoflightsinverse} is lost if one considers input outside of the faces in $\mathcal G(m;A)$.  

\begin{remark}\label{r:circleoflightsinverse}
Given a point $x \in \RR_{\ge 0}^k$, one may compute $q(x)$ using the circle of lights algorithm~\cite{wilfconjecture}, which is used to compute the Ap\'ery set of a numerical semigroup from a given generating set.  The version of the algorithm in \cite[Algorithm~7.1]{kunzfaces1} is faster, and as a byproduct computes the set of factorizations of each element of the Kunz nilsemigroup corresponding to the face of $\cC_m$ containing $q(x)$.  
\end{remark}

We next consider the cone
$$C(m;A) = \{x \in \RR_{\ge 0}^k : x_i \le c \cdot x  \text{ for all } c \in \ZZ_{\ge 0}^k \text{ with } c_1p_1 + \cdots + c_kp_k \equiv p_i \bmod m\}.$$
Despite the fact that $C(m;A)$ is defined using an infinite collection of inequalities, only finitely many are necessary.  Indeed, if $c \in \ZZ_{\ge 0}^k$ has some $c_j \ge m$, then 
$$
c \cdot x \ge (c - me_j) \cdot x
\qquad \text{for all} \qquad
x \in \RR_{\ge 0}^k,
$$
so in the definition of $C(m;A)$ one may restrict to $c$ with coordinates in $[0, m-1]$.  In~particular, this means $C(m;A)$ is a rational polyhedral cone.  

\begin{example}\label{e:3dfan}
Let $m = 13$ and $A = \{1,2,9\}$.  The Kunz fan $\mathcal G(m; A)$ has five 3-dimensional chambers, the cross sections of which are depicted in Figure~\ref{f:3dfan} alongside the Kunz nilsemigroups corresponding to each chamber.  
Any point in the interior of a chamber of $\mathcal G(m; A)$ must satisfy
\[
x_1 < 3x_9,
\qquad
x_1 < 7x_2, 
\qquad
x_2 < 2x_1,
\qquad \text{and} \qquad
x_9 < x_1 + 4x_2,
\]
which arise from the ``minimal'' ways of expressing 1, 2, or 9 in $\ZZ_{13}$ as a sum of the other two and comprise the facets of $C(m; A)$.  
% We give a more thorough description of these inequalities in Proposition~\ref{p:fanboundary}.  
\end{example}

\begin{figure}[t!]
\begin{center}
\includegraphics[width=6in]{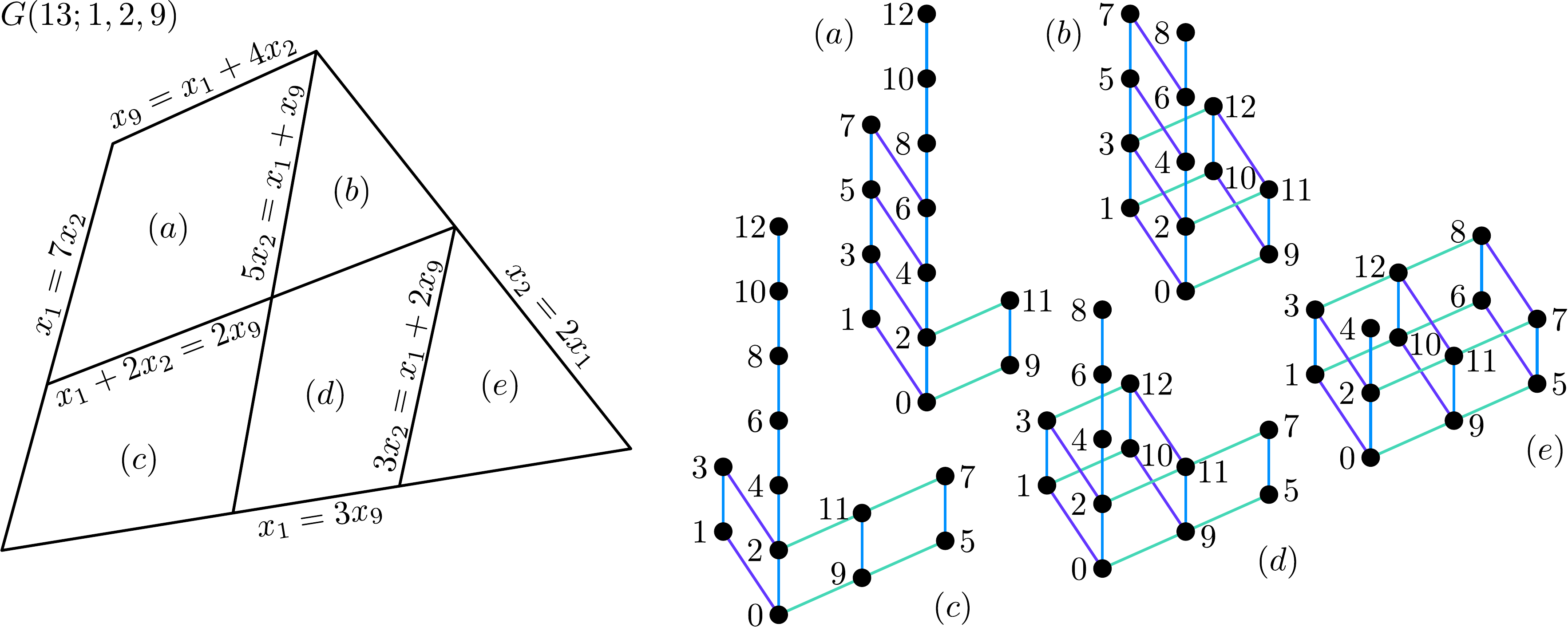}
\end{center}
\caption{Posets from the chambers of the Kunz fan $\mathcal G(13; 1,2 ,9)$.}
\label{f:3dfan}
\end{figure}

\begin{lemma}\label{l:interioratomsets}
For each $x \in C(m;A)^\circ$, the image $y = q(x)$ lies in a face $F \subseteq \cC_m$ with trivial Kunz subgroup, and the Kunz nilsemigroup $N$ of $F$ has atom set $A$.  
\end{lemma}

\begin{proof}
Since $C(m;A) \subset \RR_{\ge 0}^k$, each $x_i$ is positive, so $x$ has all positive coordinates as well.  This ensures $F$ has trivial Kunz subgroup.  Moreover, the definition of $q$ ensures $N$ has atom set contained in $A$, but if some $p_i \in A$ were not an atom of $N$, then for any $z \in \mathsf Z_N(p_i)$, we would have 
\begin{align*}
x_i
\ge y_{p_i}
&= z_1 y_{p_1} + \cdots + z_{i-1} y_{p_{i-1}} + z_{i+1} y_{p_{i+1}} + \cdots + z_k y_{p_k} \\
&= z_1 x_1 + \cdots + z_{i-1} x_{i-1} + z_{i+1} x_{i+1} + \cdots + z_k x_k,
\end{align*}
violating the assumption that $x$ lies in the interior of $C(m;A)$.  
\end{proof}

\begin{thm}\label{t:purefan}
% Boundary inequalties for fan come from when one of the atoms is no longer an atom
The cones in $\mathcal G(m;A)$ form a polyhedral fan that is pure, and the union of the cones in $\mathcal G(m;A)$ equals $C(m;A)$.  
\end{thm}

\begin{proof}
Cleary, if a face $F \subseteq \cC_m$ lies in $\mathcal F(m;A)$, then all faces of $F$ must as well, so every face of a cone in $\mathcal G(m;A)$ is a cone in $\mathcal G(m;A)$.  Moreover, for any $F, F' \in \mathcal F(m;A)$, Proposition~\ref{p:circleoflightsinverse} implies $p(F) \cap p(F') = p(F \cap F')$, so the intersection of any two cones in $\mathcal G(m;A)$ is a cone in $\mathcal G(m;A)$.  This verifies $\mathcal G(m;A)$ is a polyhedral fan.  

Next, by Lemma~\ref{l:interioratomsets}, for each $x \in C(m;A)^\circ$, $y = q(x)$ lies in a face $F \in \mathcal F(m;A)$.   
As such, $p(y) = x$ by Proposition~\ref{p:circleoflightsinverse}, which lies in $p(F) \in \mathcal G(m;A)$.  This means the union of the cones in $\mathcal G(m;A)$ contains $C(m;A)^\circ$, and since rational polyhedral cones are topologically closed, the union must equal $C(m;A)$.  

The final claim to prove is that every maximal cone in $\mathcal G(m;A)$ has dimension $k$.  Indeed, since $\mathcal G(m;A)$ contains only finitely many cones, and their union equals the $k$-dimensional cone $C(m;A)$, the union of the $k$-dimensional cones in $\mathcal G(m;A)$ must also equal $C(m;A)$.  This means every cone in $\mathcal G(m;A)$ is contained in a $k$-dimenisional cone in $\mathcal G(m;A)$.  
\end{proof}

\begin{cor}\label{c:chamberstaircase}
A face $F \in \mathcal F(m;A)$ is a chamber if and only if its Kunz nilsemigroup is staircase.   
\end{cor}

\begin{proof}
A maximal face of $\mathcal F(m;A)$ has dimension $k$ by Theorem~\ref{t:purefan}, and since the number of atoms of its Kunz nilsemigroup $N$ is also $k$, \cite[Theorem~4.3]{kunzfaces3} implies $N$~has no inner trades and therefore must be staircase.  
\end{proof}

%%%%%%%%%%%%%%%%%%%%%%%%%%%%%%%%%%%%%%%%%%%%%%%%%%%%%%%%%%%%%%%%%%%%%%%%%
\section{Outer Betti elements and facets}%%%%%%%%%%%%%%%%%%%%%%%%%%%%%%%%
\label{sec:facetouterbetti}%%%%%%%%%%%%%%%%%%%%%%%%%%%%%%%%%%%%%%%%%%%%%%
%raggedbottom%%%%%%%%%%%%%%%%%%%%%%%%%%%%%%%%%%%%%%%%%%%%%%%%%%%%%%%%%%%%

Throughout this section, let $A = \{p_1,\ldots,p_k\} \subset \ZZ_m \setminus \{0\}$ with $\gcd(A,m) = 1$, and let $N$ be a partly cancellative nilsemigroup with atom set $A$.

The main results of this section are Theorem~\ref{t:outerbettihdesc}, which identifies an $H$-description of each face $F \subseteq \cC_m$ in terms of its corresponding Kunz nilsemigroup $N$, and Theorem~\ref{t:kunzcondition}, which characterizes the finite, partly cancellative nilsemigroups that are Kunz.  

\begin{defn}\label{d:modularnilsemigroup}
A \emph{modular}%
\footnote{no relation to modular lattices}
nilsemigroup is a finite, partly cancellative nilsemigroup~$N$ together with a bijection $f : \ZZ_m \to N \setminus \{\infty\}$ such that for any $a,b\in\ZZ_m$, either $f(a) + f(b) = f(a+b)$ or $f(a) + f(b) = \infty$.  For convenience, we write $a \in \ZZ_m$ in place of $f(a) \in N$, effectively viewing $N = \ZZ_m \cup \{\infty\}$ as sets.  If $N = \<f(p_1), \ldots, f(p_k)\>$ and~$z \in \ZZ_{\ge 0}^k$, we write $\overline z = f(z_1p_1 + \cdots + z_kp_k)$, so if $z \in \mathsf Z_N(i)$ for $i \ne \infty$, then $\overline z = i$.  

A \emph{Betti equality} of a modular nilsemigroup $N$ is an equation of the form 
\[
c \cdot x = c_1x_1 + \cdots + c_kx_k = c_1'x_1+ \cdots c_k'x_k = c'\cdot x  
\]
where $c,c' \in \mathsf{Z}_N(p)$ for some $p \in N \setminus \{\infty\}$.  Let $H_N$ equal to subspace of $\RR^k$ satisfying all such equalities for a given $N$ (this coincides with the nullspace of the presentation matrix defined in \cite[Section~4]{kunzfaces3}).  
A \emph{Betti inequality} of $N$ is an inequality of the form 
\[
z \cdot x = z_1x_1 + \cdots + z_kx_k \geq a_1x_1 + \cdots + a_kx_k = a \cdot x  
\]
where $z \in \mathsf{Z}_N(\infty)$ is an outer Betti factorization and $a \in \mathsf{Z}_N(\overline{z})$.  Let $F_N \subseteq H_N$ be the rational polyhedral cone containing points $x \in H_N$ that satisfy all Betti inequalities.  
\end{defn}

\begin{example}\label{e:bettiequality}
Consider the Kunz nilsemigroup $N$ depicted on the left in Figure~\ref{f:nonkunzex}.  The trades $e_1 + e_3 \sim 2e_2$ and $e_7 + 3_8 \sim 2e_3$, occurring at $4, 6 \in N$ respectively, yield 
\[
(1,-2,1,0,0) \cdot x = 0
\qquad \text{and} \qquad
(0,0,-2,1,1) \cdot x = 0
\]
as Betti equalities, so $H_N \subseteq \RR^5$ has $\dim H_N = 3$.  Each factorization $b$ from among 
$$\begin{array}{lllll}
(1, 0, 0, 0, 1),
&
(0, 1, 0, 1, 0),
&
(-1, 1, 0, 0, 1),
&
(-1, 0, 1, 1, 0),
&
(2, -1, 0, 0, 0),
\\
(0, -1, 1, 0, 1),
&
(1, 1, -1, 0, 0),
&
(0, -1, -1, 2, 0),
&
(0, 0, 0, -1, 2),
&
(1, 0, 0, 1, -1)
\end{array}$$
yields a Betti inequality $b \cdot x \ge 0$, together defining a 3-dimensional cone $F_N \subseteq H_N$.  
\end{example}

\begin{example}\label{e:2dfan2}
Resume notation from Example~\ref{e:2dfan}.  
The interior of chamber~(b) is defined by the inequalities $7x_6 > 2x_{11}$ and $4x_{11} > 4x_6$, which can be seen as consequences of the outer Betti elements $\{(7,0)\}$ and $\{(0,4)\}$, respectively.  Note that $N$ does have a third outer Betti element, namely $\{(3,2)\}$, but the inequality $3x_6 + 2x_{11} \ge 0$ does not define a facet of chamber~(b).  
\end{example}

\begin{example}\label{e:nonkunznilsemigroup}
The modular nilsemigroup $N$ depicted in the middle in Figure~\ref{f:nonkunzex} is not Kunz.  Indeed, $N$ has no inner Betti elements, so if it were Kunz, $F_N$ would be a full dimensional cone inside $H_N = \RR^4$ by Theorem~\ref{t:outerbettihdesc}.  For a point $x \in F_N$, the inequalities 
\[
(-1,-1,1,1)\cdot x \ge 0,
\qquad
(1,-1,-1,1) \cdot x \ge 0,
\qquad \text{and} \qquad
(0,2,0,-2) \cdot x \ge 0,
\]
together imply $(0,2,0,-2) \cdot x = 0$.  In other words, every point $x \in F_N$ lies in the hyperplane $x_2 = x_4$, i.e. $F_N$ is not full dimensional in $H_N$. 
\end{example}

\begin{figure}[t!]
\begin{center}
\includegraphics[height=0.9in]{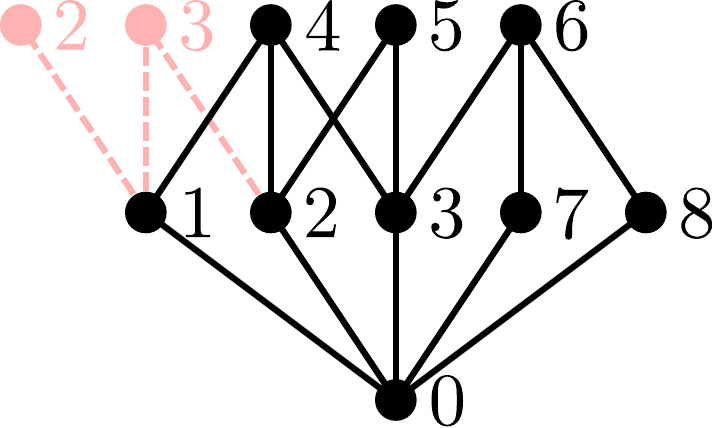}
\hspace{3em}
\includegraphics[height=0.9in]{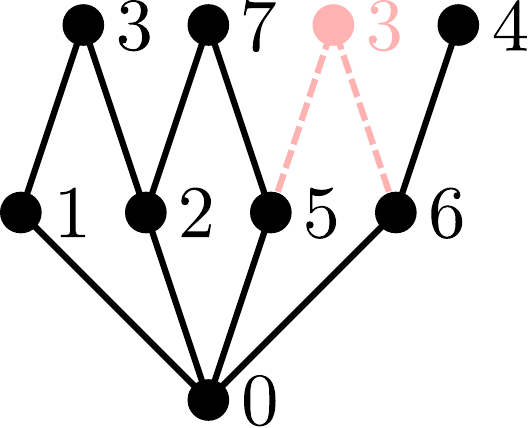}
\hspace{3em}
\includegraphics[height=0.9in]{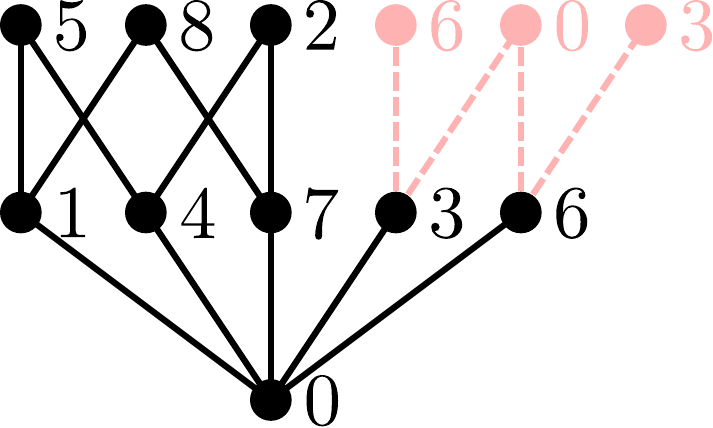}
\end{center}
\caption{Kunz posets from Examples~\ref{e:bettiequality},~\ref{e:nonkunznilsemigroup}, and~\ref{e:innerfactsminimal}}
\label{f:nonkunzex}
\end{figure}

% \begin{lemma} \label{l:obhdlemma1}
% If $N$ is a Kunz nilsemigroup corresponding to a face $F \in \cC_m$, then $p(F) \subseteq F_N$
% \end{lemma}

% \begin{proof}
% By definition $p(F) \subset H_N$, so we only need to show that points of $p(F)$ satisfy every betti inequality. Let $x \in p(F)$, and $y = q(x) \in F$. Let $z = (z_1,\ldots, z_k)$ be an outer betti factorization, and let $c = (c_1,\ldots, c_k) \in \mathsf{Z}_N(\overline{z})$. Since $z$ is an outer betti element, there is some $i$ such that $z - e_i \in \mathsf{Z}_N(p)$ for some $p \in N \setminus\{\infty\}$. Since $\overline{z} = p + p_i$, we have 
% \begin{align*}
%     z\cdot x = (z-e_i)\cdot x + e_i \cdot x = y_{p} + y_{p_i} \geq y_{p+p_i} = c \cdot x
% \end{align*}
% So $x \in p(F)$. 
% \end{proof}

In what follows, let
$$Z = \{z \in \mathsf{Z}_N(\infty) \mid z - e_i \notin \mathsf{Z}_N(\infty) \text{ for each } i \in \supp(z)\}$$
denote the set of minimal elements of $\mathsf Z_N(\infty)$ under the component-wise partial order.  

\begin{lemma}\label{l:locateouterbetti}
Suppose $N$ is a modular nilsemigroup and fix $x \in F_N^\circ$ and $y \in \mathsf{Z}_N(\infty)$.  There exist $b, c \in \ZZ_{\geq 0}^k$ such that $b$ is a factorization of an outer Betti element of $N$,
\[
y \cdot x = (b+c) \cdot x, \quad \text{and} \quad \overline{y} = \overline{b+c}.
\]
\end{lemma}

\begin{proof}
Choose $z \in Z$ so that $y = z + \ell$ for some $\ell \in \ZZ_{\geq 0}^k$.  If $z$ lies in an outer Betti element of $N$, then choosing $b = z$ and $c = \ell$ completes the proof.  Otherwise, there exists $y' \in \mathsf Z_N(\infty) \setminus Z$ such that $y' \sim z$.  Choose $z' \in Z$ so that $y' = z' + \ell'$ for some $\ell' \in \ZZ_{\geq 0}^k \setminus \{0\}$.  If $z'$ is a factorization of an outer Betti element of $N$, then choosing $b = z'$ and $c = \ell' + \ell$ completes the proof, since $y' \sim z$ ensures
$$
(b + c) \cdot x = (y' + \ell) \cdot x = (z + \ell) \cdot x = y \cdot x
\qquad \text{and} \qquad
\overline{b + c} = \overline{y' + \ell} = \overline{z + \ell} = \overline{y}.
$$
Otherwise, we may again repeat this process to obtain $z'' \in Z$.  Notice that 
\begin{align*}
(z \cdot x) - (z' \cdot x)
&= (y' \cdot x) - (z' \cdot x)
= \ell' \cdot x
\ge \min \{x_1, \ldots, x_k\},
\end{align*}
so this process must eventually terminate in a suitable choice of $b$ and $c$.  
% Suppose for the sake of contradiction that this process never terminates at some outer betti factorization $z^*$, and for each $j \geq 1$ we have $y_j = z_j + \ell_j$, with $z_j \sim y_{j+1}$ and $\ell_j \in \ZZ^{k}_{\geq 0} \setminus \{0\}$ (note that subscripts here are indicating distinct tuples, not entries of a given tuple). Since $z_j \sim y_{j+1}$, we have $z_j \cdot x = y_{j+1} \cdot x$, which means that 
% \[
% y_{j}\cdot x = y_{j+1}\cdot x + \ell_j\cdot x  
% \]
% So 
% \[ 
% (y_j - y_{j+1})\cdot x \geq X 
% \]
% i.e. the sequence $\{y_j\cdot x\}$ is a decreasing sequence where successive terms differ by a fixed minimum value. Such a sequence must eventually be less than $0$, but this is impossible since $y_j \cdot x \geq 0$ for every $j$. Thus for some $N$,
% \[
%   y \cdot x = z_N \cdot x + \left(\sum_{j=0}^{N} \ell_j \right)\cdot x
% \]
% where $z_N$ is an outer betti factorization. Letting $c = \sum_{j=0}^{n}\ell_j$, we have our desired equality. Furthermore, if $a \sim b$ then $\overline{a} = \overline{b}$, so $\overline{y} = \overline{z_N + c}$ as desired.
\end{proof}

\begin{prop}\label{p:innerfactsminimal}
Suppose $N$ is modular, fix $p \in \ZZ_m \setminus \{\infty\}$, and fix $x \in F_N$ that satisfies all Betti inequalities stricly.  If $z \in \mathsf Z_N(p)$ and $z' \in \ZZ_{\ge 0}^k$ with $\overline{z'} = p$, then
\[
z'\cdot x \geq z \cdot x,
\]
with equality if and only if $z' \in \mathsf{Z}_{N}(p)$.  
\end{prop}

\begin{proof}
If $z' \in \mathsf{Z}_N(p)$, then the definition of $F_N$ ensures $z \cdot x = z'\cdot x$, so it suffices to show that if $z' \in \mathsf{Z}_N(\infty)$ with $\overline{z'} = p$, we have $z' \cdot x > z \cdot x$.  By Lemma~\ref{l:locateouterbetti}, there is a factorization $b$ of an outer Betti element and $c \in \ZZ_{\geq 0}^k$ such that 
\[
z' \cdot x = (b + c) \cdot x
\qquad \text{and} \qquad
\overline{b+c} = p.
\]
Fixing $a \in \mathsf{Z}_N(\overline{b})$, the Betti inequality $z' \cdot x > (a+c) \cdot x$ must hold.  If $a + c \in \mathsf{Z}_N(p)$, then $(a + c) \cdot x = z \cdot x$ and we are done.  Otherwise, $(a+c) \in \mathsf{Z}_N(\infty)$, and we can again apply Lemma~\ref{l:locateouterbetti} to obtain a factorization $b'$ of an outer Betti element, $c' \in \ZZ_{\ge 0}^k$, and $a' \in \mathsf{Z}_N(\overline{b_1})$ such that 
\[
z' \cdot x > (a+c) \cdot x = (b' + c') \cdot x > (a' + c')\cdot x
\qquad \text{and} \qquad
\overline{a+c} = \overline{a' + c'} = p.
\]
As in the proof of Lemma~\ref{l:locateouterbetti}, repeating this process eventually terminates in a factorization $a'' + c'' \in \mathsf Z_N(p)$, at which point we obtain $z' \cdot x > (a'' + c'') \cdot x = z \cdot x$.  
% We can keep applying these two properties until $a_n + c_n \in \mathsf{Z}_N(p)$. Indeed, supposing that this process never terminated would yield a strictly decreasing sequence of nonnegative real numbers $\{(a_i + c_i)\cdot x\}_{i=1}^{\infty}$, but the successive differences between the terms in the sequence are bounded below by $\min_{1 \leq i \leq k} e_i \cdot x$, so such a sequence would eventually go below $0$, which is impossible. Thus, noting that $(a_n + c_n) \cdot x = z \cdot x$, we have $z' \cdot x > z \cdot x$, which concludes the proof.
\end{proof}

\begin{example}\label{e:innerfactsminimal}
The condition that $x$ satisfies all Betti inequalities strictly cannot be omitted from Proposition~\ref{p:innerfactsminimal}.  Indeed, consider the modular nilsemigroup $N$ with $m = 9$ and atom set $A = \{1,3,4,6,7\}$ depicted on the right in Figure~\ref{f:nonkunzex}.  The outer Betti elements of $N$ have factorizations
\begin{align*}
&
e_1 + e_3,
\quad
e_4 + e_3,
\quad
e_7 + e_3,
\quad
e_1 + e_6,
\quad
e_4 + e_6,
\quad
e_7 + e_6,
\\
&
2e_1,
\quad
2e_4,
\quad
2e_7,
\quad
2e_3,
\quad
2e_6, 
\quad
e_3 + e_6, 
\quad \text{and} \quad
e_1 + e_4 + e_7.
% (2,0,0,0,0),
% (0,2,0,0,0),
% (0,0,2,0,0),
% (0,0,0,2,0),
% (0,0,0,0,2),
% (0,1,0,1,0),
% (1,1,0,0,0),
% (0,1,1,0,0),
% (0,1,0,0,1),
% (1,0,0,1,0),
% (0,0,1,1,0),
% (0,0,0,1,1),
% (1,0,1,0,1)
\end{align*}
One can check the point $(x_1, x_3, x_4, x_6, x_7) = (1,3,1,4,1)$ satisfies all Betti inequalities, but $z = e_6 \in \mathsf Z_N(1)$ and $z' = 2e_1 + e_4 \in \mathsf Z_N(\infty)$ have $z \cdot x = 4 > 3 = z' \cdot x$.  
\end{example}

\begin{thm}\label{t:outerbettihdesc}
If $F \subseteq \cC_m$ is a face with Kunz nilsemigroup $N$, then $p(F) = F_N$.
\end{thm}

\begin{proof}
Proposition~\ref{p:circleoflightsinverse} implies $p(F) \subseteq F_N$.  Conversely, fix $x \in F_N$.  Since $p(F) \subseteq F_N$ and both are rational polyhedral cones, it suffices to assume $x \in F_N^\circ$.  Fix $p, q \in \ZZ_m \setminus \{0\}$ so that $p + q \ne 0$, and fix factorizations $z_1 \in \mathsf{Z}_N(p)$ and $z_2 \in \mathsf{Z}_N(q)$. We must show for any $z \in \mathsf{Z}_N(p + q)$, we have $(z_1+z_2)\cdot x \geq z \cdot x$. If $z_1 + z_2 \in \mathsf{Z}_N(p + q)$, then $(z_1 + z_2) \cdot x = z \cdot x$.  Otherwise, $z_1 + z_2 \in Z_N(\infty)$, so $(z_1 + z_2)\cdot x \geq z \cdot x$ by Lemma~\ref{p:innerfactsminimal}.  Thus $x \in p(F)$, thereby completing the proof.
\end{proof}

\begin{thm} \label{t:kunzcondition}
A modular nilsemigroup $N$ is Kunz if and only if there is a point $x \in F_N$ satisfying all Betti inequalities strictly.  
\end{thm}

\begin{proof}
If $N$ is a Kunz nilsemigroup, then there is some corresponding face $F \in \cC_m$, and by Theorem~\ref{t:outerbettihdesc}, $p(F) = F_N$.  
Since $\dim p(F) = \dim F$ by Proposition~\ref{p:circleoflightsinverse}, $p(F^\circ) = F_N^\circ$, so $p(x)$ satisfies all Betti inequalities strictly for any $x \in F^\circ$. 

Conversely, suppose $N$ is modular and there is some point $x$ that satisfies all Betti inequalities of $N$ strictly. By Proposition~\ref{p:innerfactsminimal}, $x \cdot e_i < z \cdot x$ for any $z \in \mathsf Z_N(\infty)$ with $\overline z = p_i$, so $x \in C(m;A)^\circ$.  Since the interiors of the cones of $\mathcal G(m;A)$ partition $C(m;A)^{\circ}$ by Theorem~\ref{t:purefan}, $x$ lies in the interior of $p(F)$ for some face $F \subseteq \cC_m$ whose Kunz nilsemigroup $M$ has atom set $A$. From this, we can conclude that $N = M$, since they have the same atom set and Proposition~\ref{p:innerfactsminimal} forces $Z_N(p) = Z_M(p)$ for any $p \in \ZZ_m$.  
\end{proof}

\begin{example}\label{e:kunzcondition}
Consider the nilsemigroup $N$ from Example~\ref{e:nonkunznilsemigroup}.  Even though $N$ is not Kunz, we have $F_N \in \mathcal G(8; 1,2,5,6)$.  Indeed, letting $M$ denote the modular nilsemigroup obtained from $N$ by adding $e_5 + e_6$ as a factorization of $3$, then $M$ is Kunz by Theorem~\ref{t:kunzcondition}, and one can check via computation that $F_N = F_M$.  

On the other hand, if $N$ is the nilsemigroup from Example~\ref{e:innerfactsminimal} and $A = \{1,3,4,6,7\}$, then one can verify $F_N \notin \mathcal G(9;A)$ computationally.  Indeed, $\mathcal G(9;A)$ has 6 chambers, all of which share a 3-dimensional face that is propertly contained in $F_N$.  
\end{example}

%%%%%%%%%%%%%%%%%%%%%%%%%%%%%%%%%%%%%%%%%%%%%%%%%%%%%%%%%%%%%%%%%%%%%%%%%
\section{Walking the Kunz fan}%%%%%%%%%%%%%%%%%%%%%%%%%%%%%%%%%%%%%%%%%%%
\label{sec:kunzwalk}%%%%%%%%%%%%%%%%%%%%%%%%%%%%%%%%%%%%%%%%%%%%%%%%%%%%%
%raggedbottom%%%%%%%%%%%%%%%%%%%%%%%%%%%%%%%%%%%%%%%%%%%%%%%%%%%%%%%%%%%%

Throughout this section, let $A = \{p_1,\ldots,p_k\} \subseteq \ZZ_m \setminus \{0\}$ with $\gcd(A,m) = 1$, and let $N$ be a Kunz nilsemigroup with atom set $A$.  

The goal of this section is to give an algorithm, outlined below, for computing the faces of $\cC_m$ with atom set $A$ by ``walking'' the chambers of $\mathcal G(m; A)$.  Using results in this section, every step can be completed combinatorially (i.e., in terms of nilsemigroups, without relying on polyhedral computations) with the exception of step~(3).  

\begin{enumerate}[(1)]
\item 
Locate a point $x \in \RR_{\ge 0}^k$ so that the face $F_N \in \mathcal G(m; A)$ with $x \in F_N$ has $N$ staircase.  This can be done by applying $q$ to small random perturbation of the all-1's vector, since the all-1's vector is guaranteed to lie in some face of $\mathcal G(m; A)$, and in any open $k$-dimensional ball centered there, the set of points contained in non-maximal faces of $\mathcal G(m; A)$ has measure 0.  Recall that $N$, along with its outer Betti elements, can be computed from $x$ via the circle-of-lights algorithm, as in Remark~\ref{r:circleoflightsinverse}.  

\item 
Use Theorem~\ref{t:outerbettihdesc} to obtain an $H$-description of $F_N$ with one inequality per outer Betti elements of $N$.  

\item 
Obtain an irredundant $H$-description of $F_N$ by identifying which outer Betti elements of $N$ yield supporting inequalities of $F_N$.  We were not able to accomplish this without the use of a polyhedral computation (e.g. \texttt{Normaliz}~\cite{normaliz3}); Proposition~\ref{p:reduandantbettis} and Remark~\ref{r:cutandpaste} identify sufficient and necessary conditions, respectively, but fall short of a full characterization.  

\item 
Determine which facets of $F_N$ lie on the boundary of $C(m; A)$; we identify a combinatorial method for doing so in Proposition~\ref{p:fanboundary}.  

\item 
For each remaining outer Betti element, determine the Kunz nilsemigroup $N'$ for which $F_{N'}$ shares a facet with $F_N$; we identify a combinatorial method for doing so in Theorem~\ref{t:cutandpaste}.  

\item 
Using a graph traversal algorithm (e.g., depth-first search or breadth-first search), repeat steps (2) through (5) to compute a full list of the faces of $\mathcal G(m; A)$.  

\end{enumerate}

Prior efforts to enumerate the faces of $\cC_m$ relied on computing the entire face lattice, or faces up to a particular codimension~\cite{wilfmultiplicity}.  The above allows one to compute the faces containing numerical semigroup of small embedding dimension, which are of bounded dimension.  
For instance, when verifying \cite[Conjecture~7.1]{minprescard}, we computed all faces of every Kunz fan with $k = 3$ and $m \le 42$ on our personal machines in under a day (we ran out of memory at $m = 43$), whereas computing the full face lattice for $m = 18$ for~\cite{wilfmultiplicity} took a cluster 3 weeks.  The above is also prime for parallelization, since if one fixes $m$ and $k$, each atom set $A$ can be run independently.  

The above algorithm turns out to coincide with the notion of a Gr\"obner walk~\cite{gfanpaper}; Remark~\ref{r:groebnerfans} makes this connection explicit.

\begin{defn}\label{d:irredundantinequality}
Fix a staircase Kunz nilsemigroup $N$.  We say an outer Betti element $z \in \mathsf Z_N(\infty)$ is \emph{irredundant} if its Betti inequality $z \cdot x \ge a \cdot x$ defines a facet of $F_N$, where $a \in \mathsf Z_N(\overline z)$.  
\end{defn}

We briefly identify some redundant outer Betti elements, for use in Section~\ref{sec:applications}.  

\begin{prop}\label{p:reduandantbettis}
Let $N$ be a staircase Kunz nilsemigroup.  
\begin{enumerate}[(a)]
\item 
Any two outer Betti elements $b, b' \in \mathsf Z_N(\infty)$ with $\overline b = \overline b'$ must have disjoint support.  

\item 
If $b \in \mathsf Z_N(\infty)$ is an outer Betti element and $\overline b = 0$, then $b$ is irredundant if and only if $\supp(b) = \{i\}$ and $a_1, \ldots, a_k$ are distinct modulo $\gcd(i,m)$.  

\item 
There is at most one outer Betti element of $N$ with full support, and such an outer Betti element is redundant if $|A| > 1$.  

\end{enumerate}
\end{prop}

\begin{proof}
Suppose $b_i > 0$ and $b_i' > 0$ for some $i$.  Since $b$ and $b'$ are outer Betti elements, we have $b - e_i, b' - e_i \in Z_N(\overline b - p_i)$, and since $N$ is staircase, $b - e_i = b' - e_i$ so $b = b'$.  

Next, if $b_i > 0$ and $b_j > 0$, then $b$ is redundant by the Farkas lemma, as its Betti inequality follows from $x_i \ge 0$ and $x_j \ge 0$.  As such, suppose $b_i$ is the only nonzero entry of $b$.  Writing $d = \gcd(i,m)$ for the order of $i \in \ZZ_m$, \cite[Corollary~3.7]{kunzfaces1} and Proposition~\ref{p:circleoflightsinverse} imply $C(m; A) \cap \{x \in \RR^d : x_i = 0\}$ projects faithfully onto $C(d; \overline A \setminus \{\overline 0\})$, where $\overline A \subseteq \ZZ_d$ is the set of residue classes of $a_1, \ldots, a_k$ in $\ZZ_d$.  This cone has dimension $k-1$ if and only if $a_1, \ldots, a_k$ are distinct modulo $d$, thereby proving the second claim.  

Lastly, the final claim immediately follows from the first and second.  
% the observation that any full support outer Betti element $b$ has $\overline b = 0$, so its Betti inequality follows immediately from non-negativity.  
\end{proof}

\begin{prop}\label{p:fanboundary}
Suppose $N$ is a staircase Kunz nilsemigroup and $b \in \mathsf Z_N(\infty)$ is an irredundant outer Betti element.  The~facet of $F_N$ supported by the Betti inequality for $b$ lies on the boundary of $C(m;A)$ if and only if either (i)~$\overline b \in A$, or (ii)~$\overline b = 0$.  
\end{prop}

\begin{proof}
By definition, a facet inequality of $C(m;A)$ is either of the form $x_i \ge c \cdot x$ with 
$$c_1p_1 + \cdots + c_kp_k \equiv p_i \bmod m,$$
or of the form $x_i \ge 0$.  As such, an outer Betti element whose Betti inequality is one of these two forms must fall into the claimed case~(i) or~(ii), respectively.  
\end{proof}

\begin{example}
Resume notation from Example~\ref{e:2dfan}.  The boundaries between neighboring chambers are labeled in Figure~\ref{f:2dfan} by an equality defining the linear subspace of $\RR^2$ they lie in.  By Theorem~\ref{t:cutandpaste}, each such equality induces a single trade in the Kunz nilsemigroup corresponding to that codimension 1 face.  For instance, the ray between chambers~(a) and~(b) has corresponding Kunz nilsemigroup with the trade $(7,0) \sim (0,2)$, since every point on that ray must satisfy $7x_6 \ge 2x_{11}$ from chamber~(b) and $7x_6 \le 2x_{11}$ from chamber~(a).  Additionally, the ray between chambers~(b) and~(c) has Kunz nilsemigroup with the trade $(4,0) \sim (0,4)$, and thus no numerical semigroups lie on that ray by~\cite[Corollary~3.16]{adventure5}.  
\end{example}

\begin{thm}\label{t:cutandpaste}
Suppose $\dim F_N = k - 1$ and $F_N$ is not on the boundary of $C(m;A)$.  
\begin{enumerate}[(a)]
\item 
There is a unique inner Betti element $p \in N$, and $\mathsf Z_N(p) = \{z, z'\}$.  

\item 
For any non-nil $q \in N$, there exist unique factorizations $w, w' \in \mathsf Z_N(q)$ such that $z \not\preceq w$ and $z' \not\preceq w'$, where $\preceq$ denotes the component-wise partial order on $\ZZ_{\ge 0}^k$.  Moreover, $w \ne w'$ if and only if $p \preceq q$ in $N$.  

\item 
For a staircase Kunz nilsemigroup $N'$ with $F_N \subseteq F_{N'}$, either (i) $z \not\preceq w$ for all $q \in N'$ and $w \in \mathsf Z_{N'}(q)$, or (ii) $z' \not\preceq w$ for all $q \in N'$ and $w \in \mathsf Z_{N'}(q)$.  
% Stepping through a non-boundary chamber wall does the cut-and-paste.  The poset for the interior of the wall has a single inner trade.  

\end{enumerate}
\end{thm}

\begin{proof}
Fix $c \in \ZZ^k$ with $\gcd(c) = 1$ and $H_N = \{x \in \RR^k : c \cdot x = 0\}$.  Write $c = c^+ - c^-$ where $c^+, c^- \in \ZZ_{\ge 0}^k$ have disjoint support, and let $d \in \ZZ_{\ge 1}$ be minimal with $\overline{d c^+} = \overline{d c^-}$.  
For any Betti element $p \in N$ and factorizations $z, z' \in \mathsf Z_N(p)$, we have $z - z' \in \ZZ c$.  As~such, $|\mathsf Z_N(p)| = 2$, as only one factorization of $p$ can avoid the positive (or~negative) support of $c$.  In particular, $\mathsf Z_N(p) = \{z, z'\}$ such that $z \in d \ZZ_{\ge 1} c^+$ and $z' \in d \ZZ_{\ge 1} c^-$.  To prove~(a), it suffices to show $d \overline{c^+}$ is a Betti element of $N$.  However, this follows from Proposition~\ref{p:circleoflightsinverse} and the fact that $d \overline{c^+} = d \overline{c^-} \ne 0$ by Propositions~\ref{p:reduandantbettis}(c) and~\ref{p:fanboundary} since $F_N$ is not on the boundary of $C(m; A)$.  

Having proven part~(a), let $p \in N$ denote the unique inner Betti element and write $\mathsf Z_N(p) = \{z,z'\}$.  Part~(b) then follows from part~(a), since any two factorizations of a given non-nil element $q \in N$ differ by an integer multiple of $z - z'$, and Proposition~\ref{p:circleoflightsinverse} ensures that if a factorization is preceded by $z$ or $z'$, then performing the trade $z \sim z'$ or $z' \sim z$, respectively, results in another factorization of $q$.  Lastly, for any $x \in F_{N'}^\circ$, the factorizations of any non-nil $q \in N$ are totally ordered by their dot product with~$x$ (in fact, they form an arithmetic sequence with step size $(z - z') \cdot x$), so since $N'$ is staircase it must be as prescribed in part~(c).  
\end{proof}

\begin{remark}\label{r:cutandpaste}
Theorem~\ref{t:cutandpaste} identifies a necessary condition for an outer Betti element $z$ of $N$ to be irredundant.  Indeed, one may na\"ively apply steps~(4) and~(5) for $z$ to obtain a hypothetical factorizations of each element of $\ZZ_m$.  If these do not form a valid staircase (e.g., some non-nil $p \in N$ and $z' \in \mathsf Z_N(p)$ has $z_i' > 0$ but $z' - e_i \notin \mathsf Z_N(p - a_i)$), then $z$ must be redundant.  For example, in Example~\ref{e:3dfan}, the Kunz nilsemigroup $N$ corresponding to the face in chamber~(e) has outer Betti element $z = (1,2,0)$ with $\overline z = 5$, but applying Theorem~\ref{t:cutandpaste} would move 5, 6, 7, and 8 to new locations in the Kunz poset, leave 6 ``dangling'' over the outer Betti element $(2,0,0)$.  
Geometrically, this amounts to crossing the hyperplane $x_1 + 2x_2 = 2x_9$ without first entering chamber~(d).  

Note that the condition identified in the previous paragraph is not sufficient.  Indeed, consider the Kunz nilsemigroup $N$ in Example~\ref{e:nonkunznilsemigroup}, and let $N'$ denote the nilsemigroup obtained from $N$ by replaing $\mathsf Z_{N'}(3) = \{(0,0,1,1)\}$.  One may readily check that $N$ is Kunz and has an outer Betti element $z$ above 1 and 2 with $\overline z = 3$, but applying Theorem~\ref{t:cutandpaste} to $z$ yields the modular nilsemigroup $N$, which is not Kunz.  
\end{remark}

% \begin{example}
% Cut-and-paste may require several cuts.  $m = 10$, $A = \{2,7\}$.  
% \end{example}

\begin{remark}\label{r:groebnerfans}
Under a different viewpoint, one may realize the fan $\mathcal G(m;A)$ as a portion of the Gr\"obner fan of a certain lattice ideal.  We defer the reader to \cite{grobpoly} for definitions of Gr\"obner bases and Gr\"obner fans, and~\cite{cca} for definition of lattices and lattice ideals.  

Given $m$ and $A$, consider the lattice ideal 
$$I_L = \<x^a - x^b : a - b \in L\> \subseteq \kk[x_1, \ldots, x_k]$$
for the rank $k$, index $m$ lattice
$$L = \{(z_1, \ldots, z_k) \in \ZZ^k : a_1z_1 + \cdots + a_kz_k \in m\ZZ\}.$$
Each face in $\mathcal G(m;A)$ is a face of the Gr\"obner fan of $I_L$.  Indeed, by \cite[Corollary~7.29]{cca} and the discussion preceding it, the initial ideal corresponding to each chamber $F$ of the Gr\"obner fan of $I_L$ contains precisely the monomials whose exponent vectors are not ``optimal'' with respect to a vector $x \in F^\circ$; upon unraveling definitions, this result is encoded in Proposition~\ref{p:circleoflightsinverse}.  

Gr\"obner walks, as they are called~\cite{gfanpaper}, allow one to compute the Gr\"obner fan of a given polynomial ideal in a similar fashion to the algorithm at the start of this section.  The manuscript~\cite{genericgroebnerwalk} discusses a Gr\"obner walk for a general ideal $I$, and~\cite{toricgroebnerwalk} concerns the special case where $I$ is toric (i.e., $I$ is the lattice ideal of a saturated lattice).  Note that the lattice ideal $I_L$ above is not toric, so the results in~\cite{toricgroebnerwalk} cannot be directly applied in our setting, though some of our results have analogs in~\cite{toricgroebnerwalk}.  For instance, the necessary condition in Remark~\ref{r:cutandpaste} is reminiscent of \cite[Theorem~3.6]{toricgroebnerwalk}.  Additionally, the lattice $L$ defined above has the added property that each coset of $L$ in $\ZZ^k$ naturally corresponds to an element of $\ZZ_m$, which provides the foundation of the nilsemigroup viewpoint used throughout this paper.  
\end{remark}

%%%%%%%%%%%%%%%%%%%%%%%%%%%%%%%%%%%%%%%%%%%%%%%%%%%%%%%%%%%%%%%%%%%%%%%%%
\section{Examples and applications}%%%%%%%%%%%%%%%%%%%%%%%%%%%%%%%%%%%%%%
\label{sec:applications}%%%%%%%%%%%%%%%%%%%%%%%%%%%%%%%%%%%%%%%%%%%%%%%%%
%raggedbottom%%%%%%%%%%%%%%%%%%%%%%%%%%%%%%%%%%%%%%%%%%%%%%%%%%%%%%%%%%%%

%%%%%%%%%%%%%%%%%%%%%%%%%%%%%%%%%%%%%%%%%%%%%%%%%%%%%%%%%%%%%%%%%%%%%%%%%
\subsection{Applications to open problems}%%%%%%%%%%%%%%%%%%%%%%%%%%%%%%%
\label{subsec:applications}%%%%%%%%%%%%%%%%%%%%%%%%%%%%%%%%%%%%%%%%%%%%%%
%raggedbottom%%%%%%%%%%%%%%%%%%%%%%%%%%%%%%%%%%%%%%%%%%%%%%%%%%%%%%%%%%%%

We begin with a proof of \cite[Conjecture~7.3]{minprescard}.  Following the notation from~\cite{minprescard}, given a numerical semigroup $S$, we write $\eta(S) = |\rho|$, where $\rho$ is any minimal presentation of $S$.  

\begin{defn}\label{d:refinement}
Suppose $N$ and $N'$ are Kunz nilsemigroups.  We say $N$ is a \emph{refinement} of $N'$ if for each non-nil $i \in N$, we have $\mathsf Z_{N'}(i) \subseteq \mathsf Z_N(i)$.  
\end{defn}

\begin{thm}\label{t:nablastaircasebound}
Any Kunz nilsemigroup $N$ is a refinement of some staircase Kunz nilsemigroup $N'$ with identical atom set.  Moreover, if $N$ and $N'$ have $\beta$ and $\beta'$ outer Betti elements, respectively, and $N$ has minimal presentation $\rho$, then $|\rho| + \beta \le \beta'$.  In~particular, if $S$ and $S'$ are numerical semigroups whose Kunz nilsemigroups are $N$ and $N'$, respectively, then $\eta(S) \le \eta(S'$).  
\end{thm}

\begin{proof}
Let $m = |N|$, and let $F \subseteq \cC_m$ denote the face whose Kunz nilsemigroup is $N$.  By Theorem~\ref{t:purefan}, $\mathcal G(m;A)$ is pure, so there exists $F' \subseteq \cC_m$ with $F \subseteq F'$ whose Kunz nilsemigroup $N'$ is staircase.  This means $N$ is a refinement of $N'$, and moreover, if we write $\mathsf Z_{N'}(i) = \{z\}$, then for any $y \in (F')^\circ$, $z$ is the element of $\mathsf Z_N(i)$ whose dot product with $p(y)$ is minimal.  

Now, let $B$ be an outer Betti element of $N$.  Fix $y \in (F')^\circ$, and let $z \in B$ denote the element of $B$ minimizing $z \cdot p(y)$.  We claim $\{z\}$ is an outer Betti element of $N'$.  Indeed, for each $i \in \supp(z)$, we have $(z - e_i) \in B - e_i = \mathsf Z_N(p)$ for some $p \in N$.  By~the minimality of $z$, $z - e_i$ has minimal dot product with $p(y)$ among elements of $\mathsf Z_N(p)$, so $\mathsf Z_{N'}(p) = \{z - e_i\}$.  This proves the claim.  

Next, fix a Betti element $i \in N$.  Let $\mathsf Z_{N'}(i) = \{z\}$, and let $Z \subseteq \mathsf Z_N(i)$ denote a connected component of $\nabla_i$ not containing $z$.  By an identical argument to the preceding paragraph, $\{z'\}$ is an outer Betti element of $N'$ for some $z' \in Z$.  Moreover, $z'$~cannot lie in any outer Betti element of $N$ since it lies in $\mathsf Z_N(i)$.  

We now conclude the desired inequalities hold by Theorem~\ref{t:minprescard}.  
\end{proof}

\begin{remark}\label{r:nablastaircasebound}
It remains an interesting open question to bound $\eta(S)$ in terms of $\mathsf m(S)$ and $\mathsf e(S)$.  This was posed in \cite[Conjecture~7.2]{minprescard}, and recent progress and a survey can be found in~\cite{nssyzygybounds,moscariellominpres,openproblemsrelations}.  Theorem~\ref{t:nablastaircasebound} provides an avenue for further headway.  
\end{remark}

A numerical semigroup $S$ with 4 generators can have $\eta(S)$ arbitrarily large; see~\cite{nsbettisurvey} for examples and references.  
As such, if $k = 3$, there is no upper bound to the number of outer Betti elements a staircase Kunz nilsemigroup $N$ can have.  However, as Corollary~\ref{c:facetbound} indicates, $F_N$~can have at most 6 facets in this case.  
%  numerical semigroups can have arbitrarily many Betti elements, and since nonsymmetric numerical semigroups correspond to Kunz Posets with unique expression, this would seem to imply that faces corresponding to such posets could have arbitrarily many facets. The goal of this section is to show that this is not the case, as most outer betti elements are redundant for the purposes of a minimal h-description of $F$. Note that we need not consider outer betti elements with full support, as their corresponding betti inequality is automatically redundant.

\begin{thm}\label{t:facetbound}
If $N$ is a staircase Kunz nilsemigroup, $k = 3$, and $z, z' \in \ZZ_{\geq 0}^3$ are outer Betti elements of $N$ with $\supp(z) = \supp(z') = \{1,2\}$, then $z$ or $z'$ is redundant.
\end{thm}

\begin{proof}
Let $z = (z_1,z_2,0)$, $z' = (z_1',z_2',0)$, $(0,0,z_3) \in \mathsf{Z}_N(\overline{z})$, and $(0,0,z_3') \in \mathsf{Z}_N(\overline{z'})$.  
After relabeling as necessary, we may assume 
\[ 
z_1 > z_1',
\qquad
z_2< z_2',
\qquad \text{and} \qquad
z_3>z_3'.
\]
We will show that $z'$ is redundant. Let $w = (z_1,z_2,-z_3)$, $w' = (z_1',z_2',-z_3')$, and
\[
v = w' - w = (z_1' - z_1, z_2' - z_2, z_3 - z_3').
\]
By the previous inequalities, we can decompose $v = v^+ - v^-$ where
\[
v^+ = (0, z_2' - z_2, z_3 - z_3') \in \ZZ_{\ge 0}^3
\qquad \text{and} \qquad
v^- = (z_1 - z_1', 0, 0) \in \ZZ_{\ge 0}^3.
\]
Now, since $z$ is a factorization of a non-nil element of $N$, so is $z_1e_1$, and thus 
\[
\mathsf{Z}_N(\overline{v^+}) = \{v^-\}
\qquad \text{and} \qquad
v^+ \in \mathsf{Z}_N(\infty)
\]
since $\overline{v^+} = \overline{v^-}$.  
This means $v^+ = b + \ell$, where $b$ is an outer Betti element and $\ell \in \ZZ_{\geq 0}^3$.  Clearly $\supp(b) \subseteq \{2,3\}$, and since $v_2^+e_2$ and $v_3^+ e_3$ are factorizations of non-nil elements of $N$, we in fact have $\supp(b) = \{2,3\}$.  As such, $\mathsf{Z}_N(\overline{b}) = \{b_1e_1\}$ for some $b_1 \ge 0$.  Thus,
\[
w'
% = w + v
= w + v^+ - v^-
% = w + b + \ell - v^-
= w + (-b_1, b_2, b_3) + \ell + (b_1e_1 - v^-).
\]
If $b_1e_1 - v^- \in \ZZ_{\geq 0}^3$, then we are done by the Farkas lemma since $F_N \subseteq \RR_{\ge 0}^3$.  As such, suppose $z_1 - z_1' - b_1 > 0$.  Computing equivalence classes modulo $m$, we have 
% \[
% (z_1 - z_1')p_1
% = \overline{v^+}
% = \overline b + \overline \ell
% % = b_2p_2 + b_3p_3 + \ell_2p_2 + \ell_3p_3
% = b_1p_1 + \overline \ell.
% \]
\[
(z_1 - z_1 - b_1)p_1
% = (z_1 - z_1')p_1 - \overline b
= \overline{v^-} - \overline b
= \overline{v^+} - \overline b
= \overline \ell,
\]
and since $(z_1-z_1 - b_1)e_1$ is a factorization of a non-nil element of $N$, we have $\ell \in \mathsf{Z}_N(\infty)$.  Proceeding as above, write $\ell = b' + \ell'$, where $b'$ is an outer Betti element and $\ell' \in \ZZ_{\geq 0}^3$.  By the same argument as above, $\supp(b') = \{2,3\}$ and $\mathsf{Z}_N(\overline{b'}) = \{b_1'e_1\}$, so
\[
w' = w + (-b_1, b_2, b_3) + (-b_1', b_2', b_3') + \ell' + (b_1e_1 + b_1'e_1 - v^-).
\]
We can continue applying this argument until the rightmost parenthetical lies in $\ZZ_{\ge 0}^3$, and this process will indeed terminate since $b_2, b_3, b_2', b_3', \ldots > 0$.  
% the first coordinate  must do so eventually since we are adding integer values to the first coordinate.  
\end{proof}

\begin{cor}\label{c:facetbound}
If $k = 3$, then no two irredundant outer Betti elements of a staircase Kunz nilsemigroup $N$ have identical support.  In particular, $F_N$ has at most 6 facets.  
\end{cor}

\begin{proof}
Any outer Betti element $b$ of $N$ has nonempty support, and if $\supp(b) = \{1,2,3\}$ then $b$ is redundant by Proposition~\ref{p:reduandantbettis}.  Additionally, there are exactly three outer Betti elements of $N$ with singleton support, none of which coincide, and the remaining cases are handled by Theorem~\ref{t:facetbound} after appropriate permutation of indices.  
\end{proof}

\begin{remark}\label{r:latticefacetbound}
Resuming notation and terminology from Remark~\ref{r:groebnerfans}, if $L \subseteq \ZZ^3$ is any lattice with $\ZZ^3/L \cong \ZZ_m$, then Corollary~\ref{c:facetbound} implies each chamber of the Gr\"obner fan of the lattice ideal $I_L$ has at most 6 facets.  
Indeed, choosing $(v_1, v_2, v_3) \in \ZZ^3$ whose image generates $\ZZ^3/L$, we have
$$L = \{(x_1, x_2, x_3) \in \ZZ^3 : v_1x_1 + v_2x_2 + v_3x_3 \in m\ZZ\}.$$
As such, each chamber $F$ of the Gr\"obner fan of $I_L$, there are two possibilities:
\begin{itemize}
\item 
$F \in \mathcal G(m; v_1,v_2,v_3)$, and thus has at most 6 facets by Corollary~\ref{c:facetbound}; or

\item 
the initial ideal of $I_L$ corresponding to $F$ contains a variable, and thus has at most 4 minimal generators, since its staircase has at most one minimal generator with non-singleton support by Proposition~\ref{p:reduandantbettis}.  

\end{itemize}
According to \cite[Conjecture~6.1]{groenberfanchamberfacets}, if $L$ is a saturated $d$-dimensional lattice, then there exists a bound, in terms of $d$, on the number of facets of any chamber of the Gr\"obner fan of the toric ideal $I_L$.  Though some progress has been made on this conjecture~\cite{supernormalconfigs}, a~proof remains elusive, even in the case $d = 3$.  Given the above conclusions, we pose the following generlization of this question.  
\end{remark}

\begin{question}\label{q:latticefacetbound}
Does there exist a function $\phi:\ZZ \to \ZZ$ such that for any $d$-dimensional lattice $L$, each chamber of the Gr\"obner fan of $I_L$ has at most $\phi(d)$ facets?  
\end{question}

% \begin{remark}\label{r:4facets}
We state the following conjecture, which has been verified computations for $m \le 42$, albeit with some reservation:\ \cite[Conjecture~6.2]{groenberfanchamberfacets} claimed the same was true for the chambers in the Gr\"obner fan of any toric ideal defined by a 3-dimensional lattice, but a counterexample was located soon thereafter~\cite{vertexideal}.  
% \end{remark}

\begin{conj}\label{conj:4facets}
If $k = 3$, and $N$ is staircase and Kunz, then $F_N$ has at most $4$ facets.
\end{conj}

%%%%%%%%%%%%%%%%%%%%%%%%%%%%%%%%%%%%%%%%%%%%%%%%%%%%%%%%%%%%%%%%%%%%%%%%%
\subsection{Embedding Dimension 3}%%%%%%%%%%%%%%%%%%%%%%%%%%%%%%%%%%%%%%%
\label{subsec:embed3}%%%%%%%%%%%%%%%%%%%%%%%%%%%%%%%%%%%%%%%%%%%%%%%%%%%%
%raggedbottom%%%%%%%%%%%%%%%%%%%%%%%%%%%%%%%%%%%%%%%%%%%%%%%%%%%%%%%%%%%%

Throughout this subsection, let $k = 2$.  
In what follows, we characterize the faces of $\cC_m$ containing embedding dimension $3$ numerical semigroups by classifying the possible ``shapes'' of the Kunz nilsemigroup $N$ of such a face.  
% Throughout this subsection, we assume $N$ is a staircase nilsemigroup with two atoms: $p_1$ and $p_2$. 
Any embedding dimension~3 numerical semigroup is either complete intersection with 2 minimal trades, or not complete intersection with 3 minimal trades; for more on this dichotomy, see~\cite[Chapter~10]{numerical}.  As we will show, the Kunz nilsemigroup $N$ comes in 3 varieties; 2 in the former category, and 1 in the latter category.  

We begin with the case where $N$ is staircase.  By Proposition~\ref{p:reduandantbettis}, $N$ has either 2 or~3 outer Betti elements since at most one can have full support.  
% Thus any $2$-dimensional staircase poset will have at most one ``inner corner''.
As such, the Kunz poset of $N$ can have one of two staircase ``shapes''.  
\begin{itemize}
\item 
If $N$ has 2 outer Betti elements $(a,0)$ and $(0,c)$, then $m = ac$ and its Kunz poset forms an $a \times c$ diamond.  In this case, we say $(a,c) \in \ZZ_{\ge 2}^2$ is the \emph{shape} of $N$.  Numerical semigroups with Kunz nilsemigroup $N$ are complete intersection.  

\item 
If $N$ has 3 outer Betti elements, then its Kunz poset forms a ``V'' with full support outer Betti element $(a,c)$, and its other two outer Betti elements have the form $(a + b, 0)$ and $(0, c + d)$.  In this case, we say $N$ has \emph{shape} $(a, b, c, d) \in \ZZ_{\ge 1}^4$, and 
\[
m = (a+b)(c+d) - bd.
\]
Numerical semigroups with Kunz nilsemigroup $N$ are not complete intersection, but are uniquely presented and possess 3 minimal trades \cite[Chapter~10]{numerical}.  
\end{itemize}

\begin{example}\label{e:2gen}
The three posets from Figure \ref{f:2dfan} are all staircase posets, and their shapes are $(2,0,10,0)$, $(2,2,3,4)$ and $(2,4,3,1)$, respectively.  
\end{example}

The following result implies that there exists a Kunz nilsemigroup $N$ with a given staircase shape if and only if there exists a ``filling'' with the elements of $\ZZ_m$.  
This amounts to choosing $p_1, p_2 \in \ZZ_m$ so as to ``fill'' the given staircase shape.  
One could even visualize the staircase of $N$ as a Young tableaux, wherein each box $(i,j)$ is filled with $p \in \ZZ_m$ if $\mathsf{Z}_N(p) = \{(i,j)\}$.

\begin{prop}\label{p:2genallkunz}
Any modular staircase nilsemigroup $N$ with $2$ atoms is Kunz.
\end{prop}

\begin{proof}
If $N$ has shape $(a,c)$, then $p_1$ has order $a$ or $p_2$ has order $c$, as otherwise B\'ezout's identity yields a full support outer Betti element.  Up to symmetry, assume $p_2$ has order~$c$.  The Betti inequalities thus have the form 
\[
ax_1 \ge kx_2
\qquad \text{and} \qquad
cx_2 \ge 0
\]
where $\mathsf Z_N(ap_1) = \{(0,k)\}$.  The point $(x_1,x_2) = (c,a)$ satisfies both strictly.  

If $N$ has shape $(a,b,c,d)$,
% To show that $N$ is Kunz, we will construct a point $x = (x_1,x_2)$ strictly satisfying all outer Betti inequalities, and apply Theorem~\ref{t:kunzcondition}. 
% Under the parameterization given in Definition~\ref{d:2dstaircaseparams}, 
then the irredundant outer Betti inequalities are 
\[
(a+b)x_1 > k_2x_2,
\qquad \text{and} \qquad
(c+d)x_2 > k_1x_2,
% \qquad \text{and} \qquad
% ax_1 + cx_2 > 0,
\]
for some $k_2 < c + d$ and $k_1 < a+b$,
% The third inequality is only an outer betti inequality if $b,d \neq 0$, and in any case is strictly satisfied by any nonzero $x$. 
and the point $(x_1, x_2) = (c+d, a+b)$ satisfies both strictly.  In either case, Theorem~\ref{t:kunzcondition} completes the proof.  
\end{proof}

For any $(a, c) \in \ZZ_{\ge 2}$, there exists a Kunz nilsemigroup with shape $(a,c)$; one may choose, for instance, $p_1 = 1$ and $p_2 = a$, as then each coset of the subgroup $\<a\> \subseteq \ZZ_m$ forms a ``row'' of the staircase.  More generally, up to symmetry, a choice of $p_1$ and $p_2$ fills the staircase shape if and only if $p_2$ with order $c$ and $p_1$ generates $\ZZ_m/\<a\>$.  

% It follows from the above lemma that classifying Kunz nilsemigroups in this setting is equivalent to classifying modular nilsemigroups. Since there are no inner trades to worry about, we can construct a modular nilsemigroup with a given shape by filling the boxes of Young diagrams with elements of $\ZZ_m$ where differences between horizontally adjacent boxes are $p_1$, and differences between vertically adjacent boxes are $p_2$ (and the $(0,0)$ box is $0$).

% For the remainder of this section, we assume that $b,d$ are nonzero (i.e. the staircase has $1$ inner corner).

% \begin{remark}
Given a staircase shape $(a,b,c,d)$, it follows from \cite[Section 4]{fioletal} that a choice of $p_1$ and $p_2$ fills the staircase if and only the following $4$ conditions hold:
\begin{align*}
(a+b)p_1 &\equiv dp_2 \bmod m, & (1) \\
(c+d)p_2 &\equiv bp_1 \bmod m, & (2) \\
ap_1 + cp_2 &\equiv 0 \bmod m, & (3) \\
\gcd(p_1,p_2,m) & = 1. & (4)
\end{align*}
They cite the first author's PhD dissertation for a proof that such $p_1,p_2$ exist if and only if $\gcd(a,b,c,d) = 1$. We were not able to find a readable version of this, so we've elected to include a short proof for completeness.  
% \end{remark}

\begin{thm}\label{t:2dstaircasecondition}
There exists a Kunz nilsemigroup with 3 outer Betti elements and shape $(a,b,c,d)$ if and only if $\gcd(a,b,c,d) = 1$. 
\end{thm}

\begin{proof}
Suppose that $\gcd(a,b,c,d) = 1$, and let $f : \ZZ_m^2 \to \ZZ_m^2$ be the group homomorphism given by the matrix
\[
\begin{pmatrix}
a+b & -d \\ a & c
\end{pmatrix}.
\] 
We claim that $\ker(f) \cong \ZZ_m$, and that any generator $(p_1,p_2)$ of $\ker(f)$ satisfies $(1)-(4)$. \\
To see that $\ker(f) \cong \ZZ_m$, we use the fact that $\ker(f) \cong \ker(A)$, where 
\[
A = \begin{pmatrix}
1 & 0 \\ 0 & m
\end{pmatrix}
\]
is the Smith normal form~\cite{snf} of $f$ since $\gcd(a,a+b,-d,c) = 1$ and $(a+b)c + ad = m$.  Clearly $\ker(A)$ is generated by $(0,1)$, and thus isomorphic to $\ZZ_m$.  

Now, suppose $(p_1,p_2)$ generate $\ker(f)$.  By construction, both~(1) and~(3) are satisfied, and subtracting~(1) from~(3) yields~(2).  If either $h_1 = \gcd(p_1,m)$ or $h_2 = \gcd(p_2,m)$ equal $1$, then (4) is satisfied and we are done. Otherwise, suppose both $h_1,h_2 > 1$, and set $h = \gcd(h_1,h_2)$. Notice $\tfrac{m}{h_1}p_1 = \tfrac{m}{h_2}p_2 = 0$, but this means that $\tfrac{m}{h}(p_1,p_2) = (0,0)$. Since $(p_1,p_2)$ generate $\ker(f)$, we must have $h = 1$, as desired.  

Conversely, if $\gcd(a,b,c,d) = g > 1$, then the Smith normal form of $f$ is 
\[
\begin{pmatrix}
g & 0 \\ 0 & m
\end{pmatrix},
\]
so $\ker(f) \cong \ZZ_{m/g}$ and thus for any $(p_1,p_2) \in \ker f$, we have $\tfrac{m}{g}(p_1,p_2) = 0$.  This means that $p_1,p_2$ fail to satisfy condition (4).
\end{proof}

This leaves the case where $N$ is not staircase, i.e., $\dim F_N = 1$.  By Theorem~\ref{t:purefan}, this occurs when $F_N$ is a ray on the shared boundary of two faces of $\mathcal G(m;A)$ of the form $F_{N'}$ with $N'$ staircase.  Corollary~\ref{c:2genrays} below records all such rays, at which point Theorem~\ref{t:cutandpaste} identifies the structure of $N$.  Any numerical semigroup with Kunz nilsemigroup $N$ is complete intersection since there is a trade involving 2 generators.  

\begin{cor}\label{c:2genrays}
If $N$ is staircase and has shape $(a,c)$ and $p_2 = a$, then $F_N$ has rays
\[
\vec r_1 = (1, 0)
\qquad \text{and} \qquad
\vec r_2 = (k, c),
\]
where $\mathsf Z_N(ap_1) = \{(0,k)\}$.  
If $N$ has shape $(a,b,c,d)$, then $F_N$ has rays
% the after applying the projection map from Proposition \ref{p:circleoflightsinverse} the 2 rays of the face are given by 
\[
\vec r_1 = (d, a + b)
\qquad \text{and} \qquad
\vec r_2 = (c + d, b).
\]
\end{cor}

\begin{proof}
Apply Theorem~\ref{t:outerbettihdesc} and the proofs of Proposition~\ref{p:2genallkunz} and Theorem~\ref{t:2dstaircasecondition}.  
\end{proof}

% \begin{proof}
% The first claim follows from \cite[Theorem~4.6]{kunzfaces2}.  If $N$ has shape $(a,b,c,d)$, then the defining inequalities of $F_N$ are
% \[
% (a+b)x_1 \geq dx_2
% \quad \text{and} \quad
% (c+d)x_2 \geq bx_1.
% \]
% The result immediately follows.  
% % From this description the rays of the face are clearly the nonnegative spans of $\vec{r}_1 = (d, a+b)$ and $\vec{r}_2 = (c+d,b)$, and since the face is $2$ dimensional, it has exactly two rays.  
% \end{proof}

In the latter case of Corollary~\ref{c:2genrays}, the rays of the face $F \subseteq \cC_m$ with $p(F) = F_N$ are 
\begin{align*} 
\vec r_1 &= (dx_1 + (a+b)y_1, \ldots, dx_{m-1}+(a+b)y_{m-1}) \\
\vec r_2 &=  ((c+d)x_1 + by_1, \ldots, (c+d)x_{m-1}+by_{m-1})
\end{align*}
where $(x_p,y_p)$ is the unique factorization for each nonzero non-nil $p \in N$.  
A similar construction can be done in the former case in Corollary~\ref{c:2genrays}, though this case is also addressed in \cite[Theorem~4.6]{kunzfaces2}.

%%%%%%%%%%%%%%%%%%%%%%%%%%%%%%%%%%%%%%%%%%%%%%%%%%%%%%%%%%%%%%%%%%%%%%%%%
\subsection{Cup Posets}%%%%%%%%%%%%%%%%%%%%%%%%%%%%%%%%%%%%%%%%%%%%%%%%%%
\label{subsec:cupposets}%%%%%%%%%%%%%%%%%%%%%%%%%%%%%%%%%%%%%%%%%%%%%%%%%
%raggedbottom%%%%%%%%%%%%%%%%%%%%%%%%%%%%%%%%%%%%%%%%%%%%%%%%%%%%%%%%%%%%

In \cite{wilfmultiplicity}, the authors compute the number of rays of $\cC_m$ up to $m = 21$.  This data suggests the number of rays of $\cC_m$ grows exponentially in $m$.  In this section, we explicitly construct a family of faces whose number of rays is exponential in $m$. 

For this subsection, fix $d \ge 3$, let $m = 3(d-1)$, and consider the modular nilsemigroup $N$ with $A = \{1, d, d+1, \ldots m - d, m - 1\}$ whose divisibility poset of non-nil elements has
\[
1 \lessdot 2 \lessdot \cdots \lessdot d-1 \qquad \text{and} \qquad m-1 \lessdot m-2 \lessdot \cdots \lessdot m - (d-1)
\]
as its cover relations (we call this a \emph{cup poset}; $d = 4$ is depicted on the left in Figure~\ref{f:cupposet}).  
It is not hard to show $\dim F_N = d$ and the point $(1, d-1, \ldots, d-1, 1) \in F_N^\circ$ satisfies all Betti inequalities of $N$ strictly.  In particular, $N$ is a Kunz nilsemigroup by Theorem~\ref{t:kunzcondition}.  

\begin{figure}[t!]
\begin{center}
\includegraphics[height=1.2in]{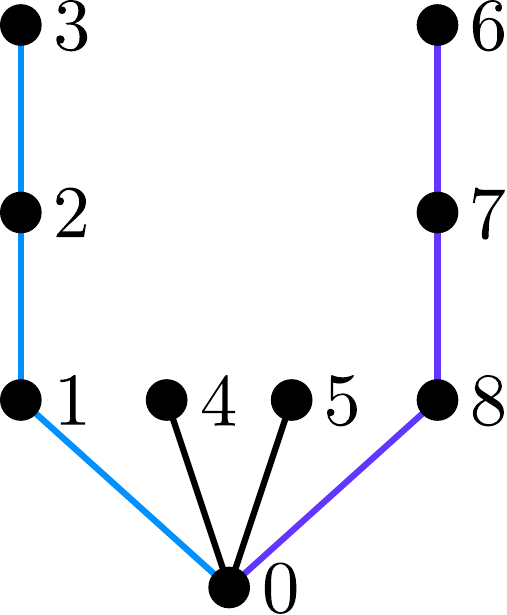}
\hspace{5em}
\includegraphics[height=1.2in]{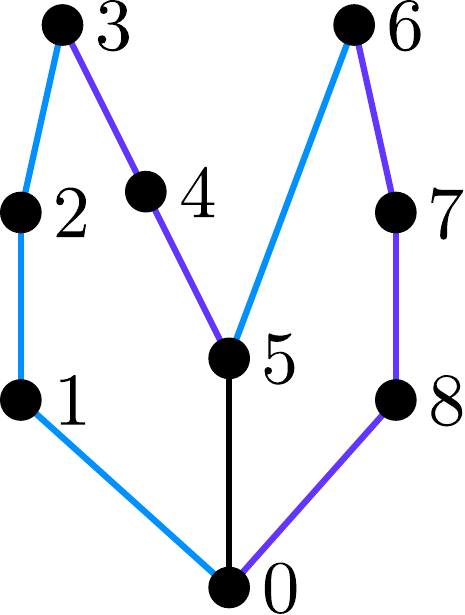}
\hspace{2em}
\includegraphics[height=1.2in]{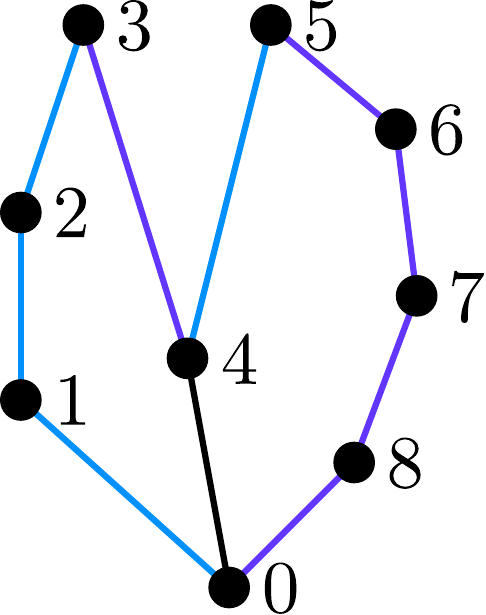}
\hspace{2em}
\includegraphics[height=1.2in]{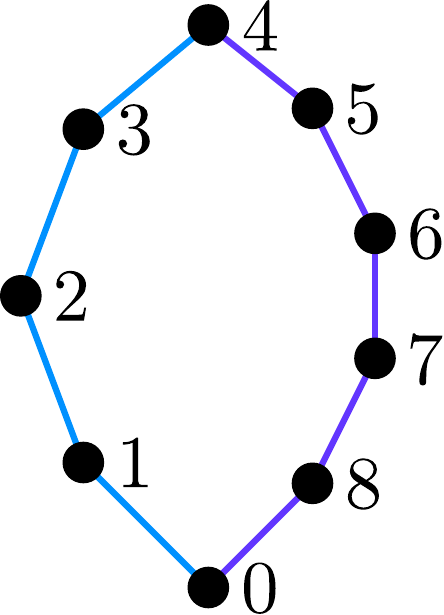}
\end{center}
\caption{Cup poset for $d = 4$, alongside the posets of three of its rays.}
\label{f:cupposet}
\end{figure}

\begin{prop}\label{p:cupposethdesc}
We have $x \in F_N$ if and only if $y = q(x)$ satisfies
\[
dy_1 \ge y_d,
\qquad
dy_{m-1} \ge y_{m-d}, 
\qquad
y_1 + y_k \ge y_{k+1},
\qquad \text{and} \qquad
y_{m-1} + y_k \ge y_{k-1}
\]
for $d \leq k \leq m - d$.  
\end{prop}

\begin{proof}
Each inequality above is a Betti inequality of $N$, so by Theorem~\ref{t:outerbettihdesc} it suffices to show the remaining outer Betti inequalities are redundant.  Each such outer Betti inequality has the form 
$y_i + y_j \ge y_{i+j}$,
where $d \le i \le j \le m - d$ and $i + j \ne 0$.  Notice
\begin{align*}
y_i + y_j &= y_i + y_j + (i - (d-1))y_{m-1} - (i - (d-1))y_{m-1} \\
& \geq (d-1)y_1 + y_j - (i - (d-1))y_{m-1} \\
& \geq \big((d-1 ) - (2(d-1) - j)\big)y_1 + (d-1)y_{m-1} - (i-(d-1))y_{m-1} \\
&= (j - (d-1))y_1 + (2(d-1) - i)y_{m-1}
\end{align*}
and
\[
y_{i+j} = \begin{cases}
(3(d-1) - (i+j))y_{m-1} & \text{if } i+j > m - d;
\\
(i+j - 3(d-1))y_1 & \text{if } i+j < d.
\end{cases}
\]
In either case the corresponding coefficient in $(j - (d-1))y_1 + (2(d-1) - i)y_{m-1}$ is larger than $y_{i,j}$, which means that $y_i + y_j \ge y_{i+j}$.  
\end{proof}

\begin{prop}\label{p:cupposetcube}
There is an invertible linear transformation $H$ that sends $F_N$ to the cone over a $(d-1)$-cube.  
\end{prop}

\begin{proof}
For $x \in F_N$, the inequalities in Proposition~\ref{p:cupposethdesc} are
\begin{align*}
dx_1 - x_2 &\ge 0,
& 
-(d-1)x_1 + x_2 + x_d &\ge 0,
\\
x_1 + x_k -x_{k+1} &\ge 0,
& 
-x_{k}  + x_{k+1} + x_{d} &\ge 0,
\\
x_1 + x_{d-1} - (d-1)x_d &\ge 0,
& 
-x_{d-1} + dx_d &\ge 0,
\end{align*}
where $k \in \{2, \ldots d - 2\}$. Let $H_1, H_2 \in \RR^{(d-1)\times d}$ denote the matrices corresponding to the first and second columns of inequalities, respectively.  Every row of the matrix $J = H_1 + H_2$ equals ${\bf j} = e_1 + e_d$.  
% since 
% \begin{align*}
% -(dx_1 - x_2) + (x_1 + x_d) &= -(d-1)x_1 + x_2 + x_d \\
% -(x_1 + x_k - x_{k+1}) + (x_1 + x_d) &= -x_{k} + x_{k+1} + x_d \\ 
% -(x_1 + x_{d-1} - (d-1)x_d ) + (x_1 + x_d) &= -x_{d-1} + dx_d 
% \end{align*}
Recall that the matrix defining the cone over the standard $(d-1)$-cube is 
\[
\begin{pmatrix}
\hphantom{-}I_{d-1} & {\bf 0} \\ -I_{d-1} & {\bf 1}
\end{pmatrix},
\]
where $\bf{1}$ is the column vector of all $1$'s. Letting 
$H = \begin{pmatrix}
H_1 \\
{\bf j}
\end{pmatrix}$, one readily checks 
\[ 
\begin{pmatrix}
\hphantom{-}I_{d-1} & {\bf 0} \\ -I_{d-1} & {\bf 1}
\end{pmatrix} H = \begin{pmatrix}
H_1 \\
-H_1 + J 
\end{pmatrix}= \begin{pmatrix}
H_1 \\
H_2
\end{pmatrix},
\]
which completes the proof.
\end{proof}
% \begin{cor}
%     If $m$ is a multiple of $2$, then $\cC_m$ has at least $2^{m-3}$ rays
% \end{cor}

\begin{remark}\label{r:cupposetrays}
The Kunz posets of the rays of $F_N$ have an interesting combinatorial structure. Picking a vertex of a cube (and thus a ray of $F_N$) is equivalent to making a binary choice for each pair of opposite faces \cite[Chapter~7]{ziegler}.  Following the map $H$ from Proposition~\ref{p:cupposetcube}, each pair of opposing faces correspond to choosing either $x_i + x_1 = x_{i+1}$ or $x_i = x_{i+1} + x_{m-1}$ for each $i = d - 1, \ldots, 2(d - 1) - 1$, yielding the Kunz poset relation $i \prec i + 1$ or $i + 1 \prec i$ for each $i$.  Three examples are depicted in Figure~\ref{f:cupposet}; 
we call these \emph{mountain range posets}.  
\end{remark}

One may use the gluing constructions in \cite[Section~6]{kunzfaces2} to construct faces of $\cC_m$ whose cross sections are simplicial.  This raises the following.  

\begin{question}\label{q:crosspolytope}
Is there a family of faces of $\cC_m$ whose cross sections are cones over cross polytopes?
\end{question}

% %%%%%%%%%%%%%%%%%%%%%%%%%%%%%%%%%%%%%%%%%%%%%%%%%%%%%%%%%%%%%%%%%%%%%%%%%
% \section*{Acknowledgements}%%%%%%%%%%%%%%%%%%%%%%%%%%%%%%%%%%%%%%%%%%%%%%
% %raggedbottom%%%%%%%%%%%%%%%%%%%%%%%%%%%%%%%%%%%%%%%%%%%%%%%%%%%%%%%%%%%%

% The authors thank Serkan Hosten and Rekha Thomas for several helpful conversations.  

%%%%%%%%%%%%%%%%%%%%%%%%%%%%%%%%%%%%%%%%%%%%%%%%%%%%%%%%%%%%%%%%%%%%%%%%%
%%%%%%%%%%%%%%%%%%%%%%%%%%%%%%%%%%%%%%%%%%%%%%%%%%%%
%%%%%%%%%%%%%%%%%%%%%%%%%%%%%%%%%%%%%%%%%%%%%%%%%%%%%%%%%%%%%%%%%%%%%%%%%

%%%%%%%%%%%%%%%%%%%%%%%%%%%%%%%%%%%%%%%%%%%%%%%%%%%%%%%%%%%%%%%%%%%%%%%%%
\end{document}